\documentclass{scrartcl}

\usepackage{amsmath}
\usepackage{amssymb}

\newcommand{\E}{\mathbb{E}}

\newcommand{\R}{\mathbb{R}}
\renewcommand{\P}{\mathbb{P}}

\renewcommand{\H}{\mathcal{H}}
\newcommand{\G}{\mathcal{G}}

\newcommand{\dom}{\textup{dom}}

\newcommand{\sub}{\subseteq}

\DeclareMathOperator*{\argmax}{arg\,max}
\DeclareMathOperator*{\argmin}{arg\,min}

\newcommand{\prox}[3][]{\operatorname{prox}^{#1}_{#2}\left(#3 \right)}

\usepackage{amsthm}
\theoremstyle{plain}
\newtheorem{theorem}{Theorem}[section]
\newtheorem{corollary}{Corollary}[section]
\newtheorem{definition}{Definition}[section]
\newtheorem{lemma}{Lemma}[section]

\newtheorem{problem}{Problem}[section]
\newtheorem{algo}{Algorithm}[section]

\theoremstyle{remark}
\newtheorem{remark}{Remark}[section]

\usepackage{mathabx}

\usepackage[utf8]{inputenc}
\usepackage{pgfplots}

\usepackage[sort,nocompress]{cite}

\usepackage{amsfonts}
\newcommand{\overbar}[1]{\mkern 1.5mu\overline{\mkern-1.5mu#1\mkern-1.5mu}\mkern 1.5mu}

\usepackage{enumerate}

\usepackage{breqn}

\usepackage{mathtools}
\usepackage{bm}
\usepackage{bbm}

\usepackage{mathrsfs}

\usepackage{leftidx}

\usepackage{graphicx}
\usepackage{caption,subcaption}

\renewcommand{\phi}{\varphi}

\newcommand{\sm}[1]{\leftidx{^{\mu_{#1}}}}

\usepackage{mathtools}
\mathtoolsset{showonlyrefs}

\title{Variable smoothing for convex optimization problems using stochastic gradients}
\author{Radu Ioan Bo\c{t}\footnote{Faculty of Mathematics, University of Vienna, Oskar-Morgenstern-Platz 1, 1090 Vienna, Austria, e-mail: radu.bot@univie.ac.at. 
		Research partially supported by FWF (Austrian Science Fund), project I 2419-N32.} \and
	Axel B\"ohm\footnote{Faculty of Mathematics, University of Vienna, Oskar-Morgenstern-Platz 1, 1090 Vienna, Austria, e-mail: axel.boehm@univie.ac.at. Research supported by the doctoral programme \textit{Vienna Graduate School on Computational Optimization (VGSCO)}, 
		FWF (Austrian Science Fund), project W 1260.}}

\date{\today}
\begin{document}

\maketitle
\begin{abstract}%
  We aim to solve a structured convex optimization problem, where a nonsmooth function is composed with a linear operator.
  When opting for full splitting schemes, usually, primal-dual type methods are employed as they are effective and also well studied.
  However, under the additional assumption of Lipschitz continuity of the nonsmooth function which is composed with the linear operator we can derive novel algorithms through regularization via the Moreau envelope.
  Furthermore, we tackle large scale problems by means of stochastic oracle calls, very similar to stochastic gradient techniques. Applications to total variational denoising and deblurring are provided.

\textbf{Keywords.} structured convex optimization problem, variable smoothing algorithm, convergence rate, stochastic gradients

\noindent \textbf{AMS Subject Classification.}  90C25, 90C15, 65Y20

\end{abstract}

\section{Introduction}%
\label{sec:introduction}

The problem at hand is the following structured convex optimization problem
\begin{equation}
  \label{eq:pd}
  \min_{x \in \H{}} f(x) + g(Kx),
\end{equation}
for real Hilbert spaces $\H$ and $\G$, $f: \H \to \overbar{\R}$ a proper, convex and lower semicontinuous function,  $g:\G \to \R$ a, possibly nonsmooth, convex and Lipschitz continuous function, and $K: \H \to \G$ a linear continuous operator.

Our aim will be to devise an algorithm for solving~\eqref{eq:pd} following the \textit{full splitting} paradigm (see\cite{pdhg,bot2014primaldualstrong,bot2015primaldualgap,vu,condat,bot2013douglas,
bot2015convergence}). In other words, we allow only proximal evaluations for simple nonsmooth functions, but no proximal evaluations for compositions with linear continuous operators, like, for instance, for $g \circ K$.

We will accomplish this feat by the means of a \textit{smoothing strategy}, which, for the purpose of this paper, means, making use of the Moreau-Yosida approximation (see~\cite{bot15variable-smoothing, bot15acceleration-smoothing, bot13double-smoothing, nesterov2007smoothing , nesterov2005smooth}). The approach can be described as follows: we ``smooth'' $g$, i.e.\ we replace it by its Moreau envelope, and solve the resulting optimization problem by an \textit{accelerated proximal-gradient algorithm} (see~\cite{chambolle_dossal, fista, nesterov83}).

The only other family of methods able to solve problems of type~\eqref{eq:pd} are the so called primal-dual algorithms, first and foremost the \textit{primal-dual hybrid gradient (PDHG)} introduced in~\cite{pdhg}. In comparison, this method does not need the Lipschitz continuity of $g$ in order to proof convergence. However, in this very general case, convergence rates can only be shown for the so-called \textit{restricted primal-dual gap} function. In order to derive from here convergence rates for the primal objective function, either Lipschitz continuity of $g$ or finite dimensionality of the problem plus the condition that $g$ must have full domain are necessary (see, for instance,~\cite[Theorem 9]{bot2015primaldualgap}). This means, that for infinite dimensional problems the assumptions required by both, PDHG and our method, for deriving convergence rates for the primal objective function are in fact equal, but for finite dimensional problems the assumption of PDHG are weaker. In either case, however, we are able to proof these rates for the sequence of iterates ${(x_{k})}_{k \geq 1}$ itself whereas PDHG only has them for the sequence of so-called \textit{ergodic iterates}, i.e.\ ${(\frac{1}{k}\sum_{i=1}^{k} x_{i})}_{k \geq 1}$, which is naturally undesirable as the averaging slows the convergence down.

Furthermore, we will also consider the case where only a stochastic oracle of the proximal operator of $g$ is available to us. This setup corresponds e.g.\ to the case where the objective function is given as
\begin{equation}
  \label{eq:finite-sum-example}
  \min_{x \in \H} f(x) + \sum_{i=1}^{m} g_{i}(K_{i}x),
\end{equation}
where, for $i=1,\dots,m$, $\G_i$ are real Hilbert spaces, $g_i : \G_i \rightarrow \R$ are convex and Lipschitz continuous functions and $K_i : \H \rightarrow \G_i$ are linear continuous operators, but the number of summands being large we wish to not compute all proximal operators of all $g_i, i=1, \dots, m$, for purpose of making iterations cheaper to compute.

For the finite sum case~\eqref{eq:finite-sum-example}, there exist algorithms of similar spirit such those in~\cite{spdhg,pesquet_repetti}. Some algorithms do in fact deal with a similar setup of stochastic gradient like evaluations, see~\cite{vu_stoch_pd}, but only for smooth terms in the objective function.

In Section~\ref{sec:preliminaries} we will cover the preliminaries about the Moreau-Yosida envelope as well as useful identities and estimates connected to it. 
In Section~\ref{sec:var-smoothing-deterministic} we will deal with the deterministic case and prove a convergence rate of $\mathcal{O}(\frac{1}{k})$ for the function values  at the iterates.
Next up, in Section~\ref{sec:var-smoothing-stochastic}, we will consider the stochastic case as described above and prove a convergence rate of $\mathcal{O}\left(\frac{\log(k)}{\sqrt{k}}\right)$.
Last but not least, we will look at some numerical examples in image processing in Section~\ref{sec:numerical_examples}.

It is important to note that the proof for the deterministic setting differs surprisingly from the one for the stochastic setting. The technique for the stochastic setting is less refined in the sense that there is no coupling between the smoothing parameter and the extrapolation parameter. Where as this technique works also works for the deterministic setting it gives a worse convergence rate of $\mathcal{O}\left(\frac{\log{k}}{k}\right)$. The tight coupling of the two sequences of parameters, however does not work in the proof of the stochastic algorithm as it does not allow for the particular choice of the smoothing parameters needed there.

\section{Preliminaries}\label{sec:preliminaries}

\begin{definition}%
  For a proper, convex and lower semicontinuous function $g: \H \rightarrow \overbar{\R}$, its convex conjugate is denoted by $g^*$ defined as a function from $\H$ to $\overbar{\R}$, given by
  \begin{equation}
    g^*(x) : = \sup_{p \in \H{}}\left\{\left\langle x, p \right\rangle  - g(p) \right\} \quad \forall x \in \H.
  \end{equation}
\end{definition}

As mentioned in the introduction, we want to \textit{smooth} a nonsmooth function by considering its Moreau envelope.
The next definition will clarify exactly what object we are talking about.
\begin{definition}%
  For a proper, convex and lower semicontinuous function $g: \H \rightarrow \overbar{\R}$, its Moreau envelope with the parameter $\mu \ge 0$ is defined as a function from $\H$ to $\R$, given by
  \begin{equation}
    \sm{}g(\cdot) := {\left( g^* + \frac{\mu}{2}\lVert \cdot \rVert^2 \right)}^*(\cdot)
    = \sup_{p \in \H}\left\{\left\langle \cdot, p \right\rangle  - g^*(p) - \frac{\mu}{2}\lVert p \rVert^2\right\}.
  \end{equation}
\end{definition}

From this definition, however, it is not completely evident that the Moreau envelope indeed fulfills its purpose in being a smooth representation of the original function. The next lemma will remedy this fact.

\begin{lemma}[see {{\cite[Proposition 12.29]{bc}}}]%
  \label{lem:grad-smooth-is-prox}
  Let $g: \H \to \overbar{\R}$ be a proper, convex and lower semicontinuous function and $\mu > 0$.
  Then its  Moreau envelope is Fr\'echet differentiable on $\H$. 
  In particular, the gradient itself is given by
  \begin{equation}
    \nabla (\sm{}g)(x) = \frac{1}{\mu}\left(x - \prox{\mu g}{x}\right) = \prox{\frac{1}{\mu}g^*}{\frac{x}{\mu}} \quad \forall x \in \H{}
  \end{equation}
  and is $\mu^{-1}$-Lipschitz continuous.
\end{lemma}
In particular, for all $\mu > 0$, a gradient step with respect to the Moreau envelope corresponds to a proximal step 
  \begin{equation}
    x - \mu \nabla (\sm{}g)(x) = \prox{\mu g}{x} \quad \forall x \in \H.
  \end{equation}
The previous lemma establishes two things. Not only does it clarify the smoothness of the Moreau envelope, but it also gives a way of computing its gradient. Obviously, a smooth representation whose gradient we would not be able to compute would not be any good.

As mentioned in the introduction, we want to smooth the nonsmooth summand of the objective function which is composed with the linear operator as this can be considered the crux of problem~\eqref{eq:pd}.
The function $g \circ K$ will be \textit{smoothed} via considering instead $\sm{}g \circ K : \H \to \R$. 
Clearly, by the chain rule, this function is continuously differentiable with gradient given for every $x \in \H$ by
\begin{equation}
  \label{eq:grad-prox}
  \nabla \left( \sm{}g \circ K \right)(x) = K^* \nabla \left( \sm{}g \right) (K x) = \frac{1}{\mu}K^*\left( Kx - \prox{\mu g}{Kx} \right) = K^* \prox{\frac{1}{\mu}g^*}{\frac{Kx}{\mu}},
\end{equation}
and is thus Lipschitz continuous with Lipschitz constant $\frac{\lVert K \rVert^2}{\mu}$.

Lipschitz continuity will play an integral role in our investigations, as can be seen by the following lemmas.
\begin{lemma}[see {\cite[Proposition 4.4.6]{borwein2010convex}}]%
  \label{lem:dom-bounded}
    Let $g:\H \to \R$ be a convex and $L_{g}$-Lipschitz continuous function. Then, the domain of its Fenchel conjugate is bounded, i.e.\
    \begin{equation}
      \dom\, g^* \sub B(0,L_{g}),
    \end{equation}
    where $B(0,L_{g})$ denotes the open  ball with radius $L_{g}$ around the origin.
\end{lemma}

\begin{lemma}%
  \label{lem:est-diff-smoothing}
  Let $g:\H \to \R$ be a convex and $L_{g}$-Lipschitz continuous function.
  For $\mu_{2} \ge \mu_{1} \ge 0$ and every $x \in \H$ it holds
  \begin{equation}
    \label{eq:est-diff-smoothing}
    g^{\mu_{2}}(x) \le g^{\mu_{1}}(x) \le g^{\mu_{2}}(x) + (\mu_{2} - \mu_{1})\frac{L_{g}^2}{2}.
  \end{equation}
\end{lemma}
\begin{proof}
  For $\mu_2 \ge \mu_1 \ge 0$ and every $x \in \H$ we have, by definition,
  \begin{equation}
    \begin{aligned}
      \sm{1}g(x) =& \sup_{p \in \dom\, g^*} \left\{\langle x,p \rangle - g^*(p) - \frac{\mu_1}{2} \lVert p \rVert^2 \right\} \\
      \le& \sup_{p \in \dom\, g^*} \left\{\langle x,p \rangle - g^*(p) - \frac{\mu_2}{2} \lVert p \rVert^2 \right\} + \sup_{p \in \dom\, g^*} \left\{\frac{\mu_2 - \mu_1}{2} \lVert p \rVert^2 \right\}\\
      \le& \sm{2}g(x) + (\mu_2 - \mu_1)\frac{L_g^2}{2},
    \end{aligned}
  \end{equation}
  where we used Lemma~\ref{lem:dom-bounded} in the last inequality.
\end{proof}

\begin{lemma}%
  \label{lem:smoothing-and-no-smoothing-lipschitz-constant}
  For $\mu\geq0$ and every $x \in \H$ it holds that
  \begin{equation}
    \sm{}g(x) \le g(x) \le \sm{}g(x) + \mu \frac{L_g^2}{2}.
  \end{equation}
\end{lemma}
\begin{proof}
  The statement follows from Lemma~\ref{lem:est-diff-smoothing} using the fact that $\leftidx{^{0}}g(x) = g(x)$.
\end{proof}

\begin{lemma}%
  \label{lem:argmax-is-gradient}
  The maximizing argument in the definition of the Moreau-Yosida envelope is given by its gradient, i.e.\ for $\mu>0$ it holds that
  \begin{equation}
    \argmax_{p \in \H}\left\{\left\langle \cdot, p \right\rangle  - g^*(p) - \frac{\mu}{2}\lVert p \rVert^2\right\} = \nabla \sm{}g(\cdot).
  \end{equation}
\end{lemma}
\begin{proof}
  Let $x \in \H$ be fixed. It holds
  \begin{equation}
    \begin{aligned}
      \argmax_{p \in \H}\left\{\langle x, p \rangle  - g^*(p) - \frac{\mu}{2}\lVert p \rVert^2\right\} =&
      \argmax_{p \in \H}\left\{-\frac{1}{2\mu}\lVert x \rVert^2 + \langle x, p \rangle  - \frac{\mu}{2}\lVert p \rVert^2 - g^*(p) \right\} \\
      =& \argmax_{p \in \H}\left\{ -\frac{\mu}{2}\left\lVert \frac{x}{\mu} - p \right\rVert^2 - g^*(p) \right\} \\
      =& \argmin_{p \in \H}\left\{ g^*(p) + \frac{\mu}{2}\left\lVert \frac{x}{\mu} - p \right\rVert^2 \right\} \\
      =& \prox{\frac{1}{\mu}g^*}{\frac{x}{\mu}}
    \end{aligned}
  \end{equation}
and the conclusion follows by using Lemma~\ref{lem:grad-smooth-is-prox}.
\end{proof}

\begin{lemma}%
  \label{lem:diff-wrt-mu}
  For a proper, convex and lower semicontinuous function $g: \H \to \overbar{\R}$  and every $x \in \H$ we can consider the mapping from $(0, +\infty)$ to $\R$ given by
  \begin{equation}
    \label{eq:map-wrt-mu}
    \mu \mapsto \sm{} g(x).
  \end{equation}
  This mapping is convex and differentiable and its derivative is given by  
  \begin{equation}
    \label{eq:map-wrt-mu-gradient}
    \frac{\partial}{\partial \mu} \sm{} g(x) = - \frac12\lVert \nabla \sm{} g(x) \rVert^2 \qquad \forall x \in \H \ \forall \mu \in (0, +\infty).
  \end{equation}
\end{lemma}
\begin{proof}
   Let $x \in \H$ be fixed. From the definition of the Moreau-Yosida envelope we can see that the mapping given in~\eqref{eq:map-wrt-mu} is a pointwise supremum of functions which are linear in $\mu$. It is therefore convex. Furthermore, since the objective function is strongly concave, this supremum
 is uniquely attained at $\nabla \sm{} g(x) = \argmax_{p \in \H} \left\{\left\langle x, p \right\rangle  - g^*(p) - \frac{\mu}{2}\lVert p \rVert^2\right\}$. According to the Danskin Theorem, the function $\mu \mapsto \sm{} g(x)$ is differentiable and its gradient is given by
  \begin{equation}
    \begin{aligned}
      \frac{\partial}{\partial \mu} \sm{} g(x) =& \frac{\partial}{\partial \mu} \sup_{p \in \H}\left\{\langle x, p \rangle  - g^*(p) - \frac{\mu}{2}\lVert p \rVert^2\right\}\\
      =& -\frac{1}{2}\lVert \nabla \sm{} g(x) \rVert^2 \quad \forall \mu \in (0, +\infty).
    \end{aligned}
  \end{equation}
\end{proof}

\begin{lemma}%
  \label{lem:two-different-smoothing-parameters}
  For $\mu_{1}, \mu_{2} > 0 $ and every $x \in \H$ it holds
  \begin{equation}
    \sm{1} g (x) \le \sm{2}g(x) + (\mu_2 - \mu_1)\frac12 \lVert \nabla \sm{1} g(x) \rVert^2.
  \end{equation}
\end{lemma}
\begin{proof}
  Let $x \in \H$ be fixed.
  Via Lemma~\ref{lem:diff-wrt-mu} we know that the map $ \mu \mapsto \sm{} g (x)$ is convex and differentiable. We can therefore use the gradient inequality to deduce that
  \begin{equation}
    \begin{aligned}
      \sm{2}g(x) \ge&\, \sm{1}g(x) + (\mu_{2} - \mu_{1}) \left( \frac{\partial}{\partial \mu} \sm{}g(x) \Bigr|_{\mu = \mu_{1}} \right) \\
    =&\, \sm{1}g(x) - (\mu_{2} - \mu_{1})\frac{1}{2} \lVert \nabla \sm{1}g(x) \rVert^2,
    \end{aligned}
  \end{equation}
which is exactly the statement of the lemma.
\end{proof}

\begin{lemma}%
  \label{lem:smoothing-and-no-smoothing}
  For $\mu > 0 $ and every $x,y \in \H$ we have that 
  \begin{equation}
    \sm{} g(x) + \left\langle \nabla \sm{}g(x), y-x \right\rangle \le g(y) - \frac{\mu}{2} \lVert \nabla \sm{}g(x) \rVert^2.
  \end{equation}
\end{lemma}
\begin{proof}
Using Lemma~\ref{lem:argmax-is-gradient} and the definition of the Moreau-Yosida envelope we get that
  \begin{equation}
    \begin{aligned}
      \sm{}g(x) + \left\langle \nabla \sm{}g(x), y - x \right\rangle =& 
      \left\langle x, \nabla \sm{}g(x) \right\rangle - g^*(\nabla\sm{}g(x)) - \frac{\mu}{2}\lVert \nabla\sm{}g(x) \rVert^2 + \left\langle \nabla\sm{}g(x), y-x \right\rangle \\
      =& \left\langle \nabla\sm{}g(x), y \right\rangle  - g^*(\nabla\sm{}g(x)) - \frac{\mu}{2}\lVert \nabla\sm{}g(x) \rVert^2 \\
      \le& \ \sup_{p \in \H}\{\left\langle p, y \right\rangle  - g^*(p) \} - \frac{\mu}{2}\lVert \nabla\sm{}g(x) \rVert^2  \\
      =& \ g(y) - \frac{\mu}{2}\lVert \nabla\sm{}g(x) \rVert^2.
    \end{aligned}
  \end{equation}
\end{proof}

In the convergence proof of Section~\ref{sec:var-smoothing-deterministic} we will need the inequality in the above lemma at the points $x := Kx$ and $y := Ky$, namely
\begin{equation}
  \label{eq:smoothing-and-no-smoothing-with-K}
  \sm{} g(Kx) + \left\langle (\nabla \sm{}g \circ K)(x), y-x \right\rangle \le g(Ky) - \frac{\mu}{2} \lVert \nabla \sm{}g(Kx) \rVert^2 \quad \forall x,y \in \H.
\end{equation}

The following lemma is a standard result for convex and Fr\'echet differentiable functions.
\begin{lemma}[see~\cite{nesterov-introductory}]%
  \label{lem:stronger-grad-inequality}
 For a convex and Fr\'echet differentiable function $h: \H \to \R$ with $L_{h}$-Lipschitz continuous gradient we have that
  \begin{equation}
    h(x) + \left\langle \nabla h(x), y-x \right\rangle \le h(y) - \frac{1}{2L_{h}} \lVert \nabla h(x) - \nabla h(y) \rVert^2 \quad \forall x,y \in \H.
  \end{equation}
\end{lemma}
By applying Lemma~\ref{lem:stronger-grad-inequality} with $h := \sm{}g$, and $x:=Kx$ and $y:=Ky$, we obtain
\begin{equation}
  \label{eq:stronger-grad-inequality-with-K}
  \sm{}g(Kx) + \left\langle \nabla (\sm{}g \circ K)(x), y-x \right\rangle \le \sm{}g(Ky) - \frac{1}{2 \mu} \lVert \nabla \sm{}g(Kx) - \nabla \sm{}g(Ky) \rVert^2 \quad \forall x,y \in \H.
\end{equation}
The following technical result will be used in the proof of the convergence statement.
\begin{lemma}%
  \label{lem:hilbert-space}
For $\alpha \in (0,1)$ and every $x,y \in \H$ we have that
  \begin{equation}
    (1 - \alpha)\lVert x - y \rVert^2  + \alpha \lVert y \rVert^2 \ge \alpha(1-\alpha)\lVert x \rVert^2.
  \end{equation}
\end{lemma}

\section{Deterministic Method}%
\label{sec:var-smoothing-deterministic}

\begin{problem}%
  \label{prob:primal_dual_lipschitz}
  The problem at hand reads
  \begin{equation}
    \min_{x \in \H{}} F(x): =f(x) + g(Kx),
  \end{equation}
  for a proper, convex and lower semicontinuous function $f: \H \to \overbar{\R}$, a convex and $L_g$-Lipschitz continuous $(L_g >0)$ function $g:\G \to \R$, and a nonzero linear continuous operator $K : \H \to \G$.
\end{problem}
The idea of the algorithm which we propose to solve~\eqref{eq:pd} is to smooth $g$ and then to solve the resulting problem by means of an accelerated proximal-gradient.

\begin{algo}[Variable Accelerated SmooThing (VAST)]%
  \label{alg:variable_smoothing_accelerated}
  Let $y_0 = x_0 \in \H, {(\mu_{k})}_{k \geq 0} \! \subseteq (0,+\infty)$, and ${(t_{k})}_{k \ge 1}$ a sequence of real numbers with $t_1=1$ and $t_{k} \ge 1$ for every $k\ge 2$. Consider the following iterative scheme
  \begin{equation*}
    (\forall k \geq 1) \quad 
    \left\lfloor \begin{array}{l}
      L_{k} = \frac{\lVert K \rVert^2}{\mu_{k}} \\
      \gamma_{k} = \frac{1}{L_{k}} \\
      x_{k} = \prox{\gamma_{k}f}{y_{k-1} - \gamma_{k} K^*\prox{\frac{1}{\mu_{k}}g^*}{\frac{K y_{k-1}}{\mu_{k}}} } \\
      y_{k} = x_{k} + \frac{t_{k}-1}{t_{k+1}}(x_{k} - x_{k-1}).
    \end{array}\right.
  \end{equation*}
\end{algo}

\begin{remark}
  The assumption $t_{1} = 1$ can be removed but guarantees easier computation and is also in line with classical choices of ${(t_{k})}_{k \geq 1}$ in~\cite{nesterov83,chambolle_dossal}.
\end{remark}
\begin{remark}
  The sequence ${(u_{k})}_{k \geq 1}$ given by
  \begin{equation}
    u_{k} := x_{k-1} + t_{k}(x_{k} - x_{k-1}) \quad \forall k \geq 1,
  \end{equation}
despite not appearing in the algorithm, will feature a prominent role in the convergence proof. Due to the convention $t_1 = 1$ we have that
  \begin{equation}
    u_{1} := x_{0} + t_{1}(x_{1} - x_{0}) = x_{1}.
  \end{equation}
We also denote
  \begin{equation}
    F^{k} = f + \sm{k}g \circ K \quad \forall k \geq 0.
  \end{equation}
\end{remark}

The next theorem is the main result of this section and it will play a fundamental role when proving a convergence rate of $\mathcal{O}(\frac{1}{k})$ for the sequence ${(F(x_k))}_{k \geq 0}$.
Similar convergence rate results for smoothing algorithms have been obtained in~\cite{bot15variable-smoothing, bot15acceleration-smoothing, bot13double-smoothing}.
The techniques used in this section, however are in flavor of~\cite{smoothing_fista}, with the most notable difference of a much easier choice of the involved parameters.

\begin{theorem}%
  \label{thm:smoothing_accelerated_1k}
  Consider the setup of Problem~\ref{prob:primal_dual_lipschitz} and let ${(x_{k})}_{k \ge 0}$ and ${(y_{k})}_{k \ge 0}$ be the sequences generated by Algorithm~\ref{alg:variable_smoothing_accelerated}.
  Assume that for every $k\ge1$
  \begin{equation}
    \mu_{k} - \mu_{k+1} - \frac{\mu_{k+1}}{t_{k+1}} \le 0
  \end{equation}
  and
  \begin{equation}
    \left( 1 - \frac{1}{t_{k+1}} \right)\gamma_{k+1} t_{k+1}^2 = \gamma_{k}t_{k}^2.
  \end{equation}
  Then, for every optimal solution $x^*$ of Problem~\ref{prob:primal_dual_lipschitz}, it holds
  \begin{equation}
    F(x_{N}) - F(x^*) \le \frac{\lVert x_{0} - x^*\rVert^2 }{2\gamma_{N} t_{N}^{2}} + \mu_{N}\frac{L_{g}^2}{2} \quad \forall N \geq 1.
  \end{equation}
\end{theorem}
The proof of this result relies on several partial results which we will prove as follows.

\begin{lemma}%
  \label{lem:double-helpful-deterministic}
  The following statement holds for every $z \in \H$ and every $k\ge 0$
  \begin{align*}
      F^{k+1}(x_{k+1}) + \frac{1}{2 \gamma_{k+1}} \lVert x_{k+1} - z \rVert^2 & \le \\
      f(z) + \sm{k+1}g(Ky_{k}) + \left\langle \nabla (\sm{k+1} g \circ K ) (y_{k}), z - y_{k}\right\rangle + \frac{1}{2 \gamma_{k+1}} \lVert z - y_{k} \rVert^2.
  \end{align*}
\end{lemma}

\begin{proof}%
  Let $k\ge 0$ be fixed. Since, by the definition of the proximal map, $x_{k+1}$ is the minimizer of a $\frac{1}{\gamma_{k+1}}$-strongly convex function we know that for every $z \in \H$
  \begin{align*}
      f(x_{k+1}) + \sm{k+1}g(Ky_{k}) + \left\langle \nabla (\sm{k+1} g \circ K ) (y_{k}), x_{k+1} - y_{k}\right\rangle + \frac{1}{2 \gamma_{k+1}} \lVert x_{k+1} - y_{k} \rVert^2 +& \\
      \frac{1}{2 \gamma_{k+1}} \lVert x_{k+1} - z \rVert^2 & \le \\
      f(z) + \sm{k+1}g(Ky_{k}) + \left\langle \nabla (\sm{k+1} g \circ K ) (y_{k}), z - y_{k}\right\rangle + \frac{1}{2 \gamma_{k+1}} \lVert z - y_{k} \rVert^2.&
    \end{align*}
  Next we use the $L_{k+1}$-smoothness of $\sm{k+1} g \circ K$ and the fact that $\frac{1}{\gamma_{k+1}} \ge L_{k+1}$ to deduce
    \begin{align*}
      f(x_{k+1}) + \sm{k+1}g(Kx_{k+1}) + \frac{1}{2 \gamma_{k+1}} \lVert x_{k+1} - z \rVert^2 &  \le \\
      f(z) + \sm{k+1}g(Ky_{k}) + \left\langle \nabla (\sm{k+1} g \circ K ) (y_{k}), z - y_{k}\right\rangle + \frac{1}{2 \gamma_{k+1}} \lVert z - y_{k} \rVert^2.&
    \end{align*}
\end{proof}

\begin{lemma}%
  \label{lem:initial}
Let $x^*$ be an optimal solution of Problem~\ref{prob:primal_dual_lipschitz}. Then it holds
  \begin{equation}
    \label{eq:initial}
    \gamma_1 (F^{1}(x_{1})- F(x^*)) + \frac12\lVert u_{1} - x^* \rVert^2 \le \frac{1}{2} \lVert x^* - x_{0} \rVert^2.
  \end{equation}
\end{lemma}
\begin{proof}
  We use the gradient inequality to deduce that for every $z \in \H$ and every $k \geq 0$
  \begin{equation}
    \sm{k+1}g(Ky_{k}) + \left\langle \nabla (\sm{k+1} g \circ K ) (y_{k}), z - y_{k}\right\rangle  \le \sm{k+1}g(Kz) \le g(Kz) 
  \end{equation}
  and plug this into the statement of Lemma~\ref{lem:double-helpful-deterministic} to conclude that
  \begin{equation}
      F^{k+1}(x_{k+1}) + \frac{1}{2 \gamma_{k+1}} \lVert x_{k+1} - z \rVert^2 \le F(z) + \frac{1}{2 \gamma_{k+1}} \lVert z - y_{k} \rVert^2.
  \end{equation}
  For $k=0$  we get that
  \begin{equation}
      F^{1}(x_{1}) + \frac{1}{2 \gamma_{1}} \lVert x_{1} - x^* \rVert^2 \le F(x^*) + \frac{1}{2 \gamma_{1}} \lVert x^* - y_{0} \rVert^2.
  \end{equation}
  Now we us the fact that $u_1 = x_1$ and $y_0= x_0$ to obtain the conclusion.
\end{proof}

\begin{lemma}%
  \label{lem:descent}
Let $x^*$ be an optimal solution of Problem~\ref{prob:primal_dual_lipschitz}.  The following descent-type inequality holds for every $k \ge 0$
  \begin{equation}
    \begin{aligned}
      F^{k+1}(x_{k+1}) - F(x^*) & +  \frac{\lVert u_{k+1} - x^* \rVert^2}{2 \gamma_{k+1} t_{k+1}^2}
      \le \left(1- \frac{1}{t_{k+1}}\right) \left(F^{k}(x_{k}) - F(x^*)\right) + \frac{\lVert u_{k} - x^* \rVert^2}{2 \gamma_{k+1} t_{k+1}^2} \\
      &+ \left(1- \frac{1}{t_{k+1}}\right)\left(\mu_{k} - \mu_{k+1}- \frac{\mu_{k+1}}{t_{k+1}}\right)\lVert \nabla\sm{k+1}g(Kx_{k}) \rVert^2. 
    \end{aligned}
  \end{equation}
\end{lemma}
\begin{proof}
 Let be $k \geq 0$ fixed. We apply Lemma~\ref{lem:double-helpful-deterministic} with $z := \left( 1- \frac{1}{t_{k+1}} \right)x_{k} + \frac{1}{t_{k+1}} x^*$ to deduce that
  \begin{equation}
    \begin{aligned}
      F^{k+1}(x_{k+1})& + \frac{\lVert u_{k+1} - x^* \rVert^2}{2 \gamma_{k+1} t_{k+1}^2}  \le f\left(\left(1- \frac{1}{t_{k+1}}\right) x_{k} + \frac{1}{t_{k+1}} x^* \right) + \sm{k+1}g(Ky_{k}) \\ 
      &+ \left(1- \frac{1}{t_{k+1}}\right) \left\langle \nabla (\sm{k+1} g \circ K ) (y_{k}), x_{k} - y_{k}\right\rangle \\ 
& + \frac{1}{t_{k+1}} \left\langle \nabla (\sm{k+1} g \circ K ) (y_{k}), x^* - y_{k}\right\rangle + \frac{1}{2 \gamma_{k+1} t_{k+1}^2} \lVert u_{k} - x^* \rVert^2.
    \end{aligned}
  \end{equation}
  Using the convexity of $f$ gives
  \begin{equation}
    \label{eq:cool1-basic}
    \begin{aligned}
      F^{k+1}(x_{k+1})& + \frac{\lVert u_{k+1} - x^* \rVert^2}{2 \gamma_{k+1} t_{k+1}^2}
      \le 
      \left(1- \frac{1}{t_{k+1}}\right) f(x_{k}) + \frac{1}{t_{k+1}}f(x^*) \\ 
      &+ \left(1- \frac{1}{t_{k+1}}\right)\sm{k+1}g(Ky_{k}) +
      \left(1- \frac{1}{t_{k+1}}\right) \left\langle \nabla (\sm{k+1} g \circ K ) (y_{k}), x_{k} - y_{k}\right\rangle \\
      &+ \frac{1}{t_{k+1}} \sm{k+1}g(Ky_{k}) +
      \frac{1}{t_{k+1}} \left\langle \nabla (\sm{k+1} g \circ K ) (y_{k}), x^* - y_{k}\right\rangle + \frac{\lVert u_{k} - x^* \rVert^2}{2 \gamma_{k+1} t_{k+1}^2}.
    \end{aligned}
  \end{equation}
  Now, we use~\eqref{eq:smoothing-and-no-smoothing-with-K} to deduce that
  \begin{equation}
    \label{eq:apply-smoothing-and-no-smoothing-basic}
\begin{aligned}
    \frac{1}{t_{k+1}} \sm{k+1}g(Ky_{k}) + \frac{1}{t_{k+1}} \left\langle \nabla (\sm{k+1} g \circ K ) (y_{k}), x^* - y_{k}\right\rangle  & \le \\
    \frac{1}{t_{k+1}} g(Kx^*) - \frac{1}{t_{k+1}}\frac{\mu_{k+1}}{2} \lVert \nabla \sm{k+1}g (Ky_{k}) \rVert^2 &
\end{aligned}  
\end{equation}
  and~\eqref{eq:stronger-grad-inequality-with-K} to conclude that
  \begin{equation}
    \label{eq:apply-stronger-grad-ineq-for-smoothed-basic}
    \begin{aligned}
      \left(1- \frac{1}{t_{k+1}}\right)\sm{k+1}g(Ky_{k}) +
      \left(1- \frac{1}{t_{k+1}}\right) \left\langle \nabla (\sm{k+1} g \circ K ) (y_{k}), x_{k} - y_{k}\right\rangle & \le \\
      \left(1- \frac{1}{t_{k+1}}\right) \sm{k+1}g(Kx_{k})  - \left(1- \frac{1}{t_{k+1}}\right)\frac{\mu_{k+1}}{2}\lVert \nabla \sm{k+1}g (Ky_{k}) - \nabla \sm{k+1} g (Kx_{k}) \rVert^2. &
    \end{aligned}
  \end{equation}
  Combining~\eqref{eq:cool1-basic},~\eqref{eq:apply-smoothing-and-no-smoothing-basic} and~\eqref{eq:apply-stronger-grad-ineq-for-smoothed-basic} gives
  \begin{equation}
    \label{eq:cool2}
    \begin{aligned}
      F^{k+1}(x_{k+1})& + \frac{\lVert u_{k+1} - x^* \rVert^2}{2 \gamma_{k+1} t_{k+1}^2}
      \le \left(1- \frac{1}{t_{k+1}}\right) \sm{k+1}g(Kx_{k}) + \left(1- \frac{1}{t_{k+1}}\right) f(x_{k}) \\
      &+ \frac{1}{t_{k+1}} g(Kx^*) + \frac{1}{t_{k+1}}f(x^*) \\ 
      &- \left(1- \frac{1}{t_{k+1}}\right)\frac{\mu_{k+1}}{2}\lVert \nabla \sm{k+1}g(Ky_{k})  - \nabla \sm{k+1} g(Kx_{k}) \rVert^2\\
 & - \frac{1}{t_{k+1}}\frac{\mu_{k+1}}{2} \lVert \nabla \sm{k+1}g(Ky_{k}) \rVert^2 + \frac{\lVert u_{k} - x^* \rVert^2}{2 \gamma_{k+1} t_{k+1}^2}.
    \end{aligned}
  \end{equation}
  The first term on the right hand side is $\sm{k+1}g(Kx_{k})$ but we would like it to be $\sm{k}g(Kx_{k})$. Therefore we use Lemma~\ref{lem:two-different-smoothing-parameters} to deduce that
  \begin{equation}
    \label{eq:cool3-basic}
    \begin{aligned}
      F^{k+1}(x_{k+1})& + \frac{\lVert u_{k+1} - x^* \rVert^2}{2 \gamma_{k+1} t_{k+1}^2}
      \le \left(1- \frac{1}{t_{k+1}}\right) \sm{k}g(Kx_{k}) + \left(1- \frac{1}{t_{k+1}}\right) f(x_{k}) \\
      &+ \frac{1}{t_{k+1}} g(Kx^*) + \frac{1}{t_{k+1}}f(x^*) + \left(1- \frac{1}{t_{k+1}}\right)(\mu_{k} - \mu_{k+1})\frac12\lVert \nabla \sm{k+1}g(Kx_{k}) \rVert^2 \\ 
      &- \left(1- \frac{1}{t_{k+1}}\right)\frac{\mu_{k+1}}{2}\lVert \nabla \sm{k+1}g(Ky_{k}) - \nabla \sm{k+1}g(Kx_{k}) \rVert^2 \\ 
& - \frac{1}{t_{k+1}}\frac{\mu_{k+1}}{2} \lVert \nabla \sm{k+1}g(Ky_{k}) \rVert^2 + \frac{\lVert u_{k} - x^* \rVert^2}{2 \gamma_{k+1} t_{k+1}^2}.
    \end{aligned}
  \end{equation}
  Next up we want to estimate all the norms of gradients by using Lemma~\ref{lem:hilbert-space} which says that
  \begin{equation}
    \label{eq:hilbert-space-applied-basic}
\begin{aligned}
      \left( 1 - \frac{1}{t_{k+1}} \right) \lVert \nabla \sm{k+1}g(Ky_{k}) -  \nabla \sm{k+1}g(Kx_{k})\rVert^2 
      + \frac{1}{t_{k+1}} \lVert \nabla \sm{k+1}g(Ky_{k}) \rVert^2&  \ge \\
      \left( 1 - \frac{1}{t_{k+1}} \right)\frac{1}{t_{k+1}} \lVert \nabla \sm{k+1}g(Kx_{k}) \rVert^2. &
\end{aligned}
  \end{equation}
  Combining~\eqref{eq:cool3-basic} and~\eqref{eq:hilbert-space-applied-basic} gives
  \begin{equation}
    \label{eq:cool4}
    \begin{aligned}
      F^{k+1}(x_{k+1})& + \frac{\lVert u_{k+1} - x^* \rVert^2}{2 \gamma_{k+1} t_{k+1}^2}
      \le \left(1- \frac{1}{t_{k+1}}\right) \sm{k}g(Kx_{k}) + \left(1- \frac{1}{t_{k+1}}\right) f(x_{k}) \\
      &+ \frac{1}{t_{k+1}} g(Kx^*) + \frac{1}{t_{k+1}}f(x^*) + \left(1- \frac{1}{t_{k+1}}\right)(\mu_{k} - \mu_{k+1})\frac12\lVert \nabla \sm{k+1}g(Kx_{k}) \rVert^2 \\ 
      &- \frac{\mu_{k+1}}{2}\left( 1 - \frac{1}{t_{k+1}} \right)\frac{1}{t_{k+1}} \lVert \nabla \sm{k+1}g(Kx_{k}) \rVert^2 + \frac{\lVert u_{k} - x^* \rVert^2}{2 \gamma_{k+1} t_{k+1}^2}.
    \end{aligned}
  \end{equation}
  Now we combine the two terms containing $\lVert \nabla \sm{k+1}g(Kx_{k}) \rVert^2$ and get that
  \begin{equation}
    \label{eq:cool5}
    \begin{aligned}
      F^{k+1}(x_{k+1})& + \frac{\lVert u_{k+1} - x^* \rVert^2}{2 \gamma_{k+1} t_{k+1}^2}
      \le \left(1- \frac{1}{t_{k+1}}\right) \sm{k}g(Kx_{k}) + \left(1- \frac{1}{t_{k+1}}\right) f(x_{k}) \\
      &+ \frac{1}{t_{k+1}} g(Kx^*) + \frac{1}{t_{k+1}}f(x^*) + \frac{\lVert u_{k} - x^* \rVert^2}{2 \gamma_{k+1} t_{k+1}^2} \\
      &+ \left(1- \frac{1}{t_{k+1}}\right)\left(\mu_{k} - \mu_{k+1}- \frac{\mu_{k+1}}{t_{k+1}}\right) \frac12 \lVert \nabla \sm{k+1}g(Kx_{k}) \rVert^2. 
    \end{aligned}
  \end{equation}
  By subtracting $F(x^*) = f(x^*) + g(Kx^*)$ on both sides we finally obtain
  \begin{equation}
    \begin{aligned}
      F^{k+1}(x_{k+1}) &- F(x^*) + \frac{\lVert u_{k+1} - x^* \rVert^2}{2 \gamma_{k+1} t_{k+1}^2}
      \le \left(1- \frac{1}{t_{k+1}}\right) \left(F^{k}(x_{k}) - F(x^*)\right) + \frac{\lVert u_{k} - x^* \rVert^2}{2 \gamma_{k+1} t_{k+1}^2} \\
      &+ \left(1- \frac{1}{t_{k+1}}\right)\left(\mu_{k} - \mu_{k+1}- \frac{\mu_{k+1}}{t_{k+1}}\right) \frac12 \lVert \nabla \sm{k+1}g(Kx_{k}) \rVert^2. 
    \end{aligned}
  \end{equation}
\end{proof}

Now we are in the position to prove Theorem~\ref{thm:smoothing_accelerated_1k}.

\begin{proof}[Proof of Theorem~\ref{thm:smoothing_accelerated_1k}] We start with the statement of Lemma~\ref{lem:descent} and use the assumption that
  \begin{equation}
    \mu_{k} - \mu_{k+1} - \frac{\mu_{k+1}}{t_{k+1}} \le 0
  \end{equation}
  to make the last term in the inequality disappear for every $k \geq 0$
  \begin{equation}
    \begin{aligned}
      F^{k+1}(x_{k+1}) &- F(x^*) + \frac{\lVert u_{k+1} - x^* \rVert^2}{2 \gamma_{k+1} t_{k+1}^2} \!
      \le \left(1- \frac{1}{t_{k+1}}\right) \left(F^{k}(x_{k}) - F(x^*)\right) + \frac{\lVert u_{k} - x^* \rVert^2}{2 \gamma_{k+1} t_{k+1}^2}.
    \end{aligned}
  \end{equation}
  Now we use the assumption that
  \begin{equation}
    \left( 1 - \frac{1}{t_{k+1}} \right)\gamma_{k+1} t_{k+1}^2 = \gamma_{k}t_{k}^2
  \end{equation}
  to get that for every $k \geq 0$
  \begin{equation}
    \label{eq:new-sum-it-up}
    \gamma_{k+1}t_{k+1}^2(F^{k+1}(x_{k+1}) - F(x^*)) + \frac{\lVert u_{k+1} - x^* \rVert^2}{2} \le 
    \gamma_{k}t_{k}^2 (F^{k}(x_{k}) - F(x^*)) + \frac{\lVert u_{k} - x^* \rVert^2}{2}.
  \end{equation}
  Let be $N \geq 2$. Summing~\eqref{eq:new-sum-it-up} from $k=1$ to $N-1$ and getting rid of the nonnegative term $\lVert u_{N} - x^* \rVert^2$ gives
  \begin{equation}
    \gamma_{N} t_{N}^{2}(F^{N}(x_{N}) - F(x^*) ) \le \gamma_{1}(F^{1}(x_{1}) - F(x^*)) + \frac{\lVert u_{1}- x^* \rVert^2}{2} \quad \forall N \geq 2. 
  \end{equation}
Since $t_1=1$, the above inequality is fulfilled also for $N=1$. Using Lemma~\ref{lem:initial} shows that
  \begin{equation}
    F^{N}(x_{N}) - F(x^*) \le \frac{\lVert x_{0} - x^*\rVert^2 }{\gamma_{N} t_{N}^{2}} \quad \forall N \geq 1.
  \end{equation}
  The above inequality, however, is still in terms of the smoothed objective function.
  In order to go to the actual objective function we apply Lemma~\ref{lem:smoothing-and-no-smoothing-lipschitz-constant} and deduce that
  \begin{equation}
    F(x_{N}) - F(x^*) \le  F^{N}(x_{N}) - F(x^*) + \mu_{N}\frac{L_{g}^2}{2} \le \frac{\lVert x_{0} - x^*\rVert^2}{2\gamma_{N} t_{N}^{2}} + \mu_{N}\frac{L_{g}^2}{2} \quad \forall N \geq 1.
  \end{equation}
\end{proof}

\begin{corollary}%
  \label{cor:param-choices-deterministic}
  There exists a choice of parameters ${(\mu_{k})}_{k\ge1}, {(t_{k})}_{k\ge1}, {(\gamma_{k})}_{k\ge1}$ such that
  \begin{equation}
    \label{eq:mu-inequality}
    \mu_{k} - \mu_{k+1} - \frac{\mu_{k+1}}{t_{k+1}} \le 0
  \end{equation}
  and
  \begin{equation}
    \label{eq:update-mu}
    \left( 1 - \frac{1}{t_{k+1}} \right)\gamma_{k+1} t_{k+1}^2 = \gamma_{k}t_{k}^2
  \end{equation}
  are for every $k \geq 1$ fulfilled and are e.g.\ given by
  \begin{equation}
   t_1=1, \quad \mu_1 = b\lVert K \rVert^2, \ \text{for} \  b >0,
  \end{equation}
  and for every $k \ge 1$
  \begin{equation}
    \label{eq:param-choices}
    t_{k+1} := \sqrt{t_{k}^2 + 2 t_{k}}, \quad 
    \mu_{k+1} := \mu_k \frac{t_{k}^2}{t_{k+1}^2 - t_{k+1}}, \quad \gamma_k := \frac{\mu_k}{\|K\|^2}.
  \end{equation}
  For this choice of the parameters we have that
  \begin{equation}
    F(x_{N}) - F(x^*) \le \frac{\lVert x_{0} - x^*\rVert^2 }{b (N+1)} + \frac{bL_{g}^2\lVert K \rVert^2}{(N+1)} \exp\left(\frac{4 \pi^2}{6}\right) \quad \forall N\ge1.
  \end{equation}
\end{corollary}
\begin{proof}
 Since $\gamma_{k}$ and $\mu_{k}$ are a scalar multiple of each other~\eqref{eq:update-mu} is equivalent to 
  \begin{equation}
    \left( 1 - \frac{1}{t_{k+1}} \right)\mu_{k+1} t_{k+1}^2 = \mu_{k} t_{k}^2 \quad \forall k \geq 1
  \end{equation}
and further to (by taking into account that $t_{k+1} > 1$ for every $k \geq 1$)    
  \begin{equation}
    \label{eq:recursion-mu}
    \mu_{k+1}  = \mu_{k} \frac{t_{k}^2}{t_{k+1}^2}\frac{t_{k+1}}{t_{k+1} - 1} = \mu_{k} \frac{t_{k}^2}{t_{k+1}^2 - t_{k+1}} \quad \forall k \geq 1.
  \end{equation}
  Our update choice in~\eqref{eq:param-choices} for the sequence ${(\mu_{k})}_{k\ge1}$ is exactly chosen in such a way that it satisfies this.
  Plugging~\eqref{eq:recursion-mu} into~\eqref{eq:mu-inequality} gives for every $k \geq 1$ the condition 
  \begin{equation}
    \label{eq:combined-equation-for-t}
    1 \le \left( 1 + \frac{1}{t_{k+1}} \right) \frac{t_{k}^2}{t_{k+1}^2}\frac{t_{k+1}}{t_{k+1} - 1} = \frac{t_{k}^2}{t_{k+1}^2} \frac{t_{k+1} +1}{t_{k+1} - 1},
 \end{equation}
 which is equivalent to
 \begin{equation}
   0 \ge t_{k+1}^3 - t_{k+1}^2 - t_{k}^2t_{k+1} - t_{k}^2
 \end{equation}
 and further to
 \begin{equation}
   t_{k+1}^2 + t_{k}^2 \ge  t_{k+1}\left(t_{k+1}^2 -  t_{k}^2\right).
 \end{equation}
 Plugging in $t_{k+1} = \sqrt{t_{k}^2+2t_{k}}$ we get that this equivalent to
 \begin{equation}
   t_{k+1}^2 + t_{k}^2 \ge  t_{k+1}2t_{k} \quad \forall k \geq 1,
 \end{equation}
 which is evidently fulfilled.
 Thus, the choices in~\eqref{eq:param-choices} are indeed feasible for our algorithm.

 Now we want to prove the claimed convergence rates. Via induction we show that 
 \begin{equation}
   \label{eq:inequalities-t}
   \frac{k+1}{2} \le t_{k} \le k \quad \forall k \geq 1.
 \end{equation}
 Evidently, this holds for $t_1=1$. Assuming that~\eqref{eq:inequalities-t} holds for $k \geq 1$, we easily see that
 \begin{equation}
   t_{k+1} = \sqrt{t_{k}^2 + 2t_{k}} \le \sqrt{k^2 + 2k} \le \sqrt{k^2 +2k +1} = k+1
 \end{equation}
 and, on the other hand,
 \begin{equation}
   t_{k+1} = \sqrt{t_{k}^2 + 2t_{k}} \ge \sqrt{\frac{{(k+1)}^2}{4} + k+1} = \frac12 \sqrt{k^2 + 6k + 5} \ge \frac12 \sqrt{k^2 + 4k + 4} = \frac{k+2}{2}.
 \end{equation}

 In the following we prove a similar estimate for the sequence ${(\mu_{k})}_{k \geq 1}$. 
To this end we show, again by induction, the following recursion for every $k\ge 2$
 \begin{equation}
   \label{eq:recursion-assumption-mu}
   \mu_{k} = \mu_1\frac{\prod_{j=1}^{k-1} t_j}{\prod_{j=2}^{k}(t_j - 1)} \frac{1}{t_{k}}.
 \end{equation}
 For $k=2$ this follows from the definition~\eqref{eq:recursion-mu}.
 Assume now that~\eqref{eq:recursion-assumption-mu} holds for $k \geq 2$. From here we have that
 \begin{equation}
   \mu_{k+1} = \mu_{k} \frac{t_{k}^2}{t_{k+1}(t_{k+1}-1)} = \mu_1\frac{\prod_{j=1}^{k-1} t_j}{\prod_{j=2}^{k}(t_j - 1)} \frac{1}{t_{k}}\frac{t_{k}^2}{t_{k+1}(t_{k+1}-1)} = \mu_1\frac{\prod_{j=1}^{k} t_j}{\prod_{j=2}^{k+1}(t_j - 1)} \frac{1}{t_{k+1}}.
 \end{equation}
 Using~\eqref{eq:recursion-assumption-mu} together with~\eqref{eq:inequalities-t} we can check that for every $k \geq 1$
 \begin{equation}
   \begin{aligned}
     \mu_{k+1} = \mu_1\frac{\prod_{j=1}^{k} t_j}{\prod_{j=2}^{k+1}(t_j - 1)} \frac{1}{t_{k+1}} = \frac{\mu_1}{t_{k+1}} \prod_{j=1}^{k}\frac{t_j}{(t_{j+1} - 1)} \ge \frac{\mu_1}{t_{k+1}} = b \lVert K \rVert^2\frac{1}{t_{k+1}},
   \end{aligned}
 \end{equation}
 where we used in the last step the fact that $t_{k+1} \le t_{k}+1$.

 The last thing to check is the fact that $\mu_{k}$ goes to zero like $\frac1k$.
 First we check that for every $k\ge1$
 \begin{equation}
   \label{eq:magic}
   \frac{t_{k}}{t_{k+1} - 1} \le 1 + \frac{1}{t_{k+1}(t_{k+1} - 1)}.
 \end{equation}
 This can be seen via
 \begin{equation}
   (t_{k} + 1) t_{k+1} \le {(t_{k} + 1)}^2 = t_{k+1}^2 + 1 \quad \forall k \geq 1.
 \end{equation}
 By bringing $t_{k+1}$ to the other side we get that
 \begin{equation}
   t_{k+1}t_{k} \le t_{k+1}^2 - t_{k+1} + 1,
 \end{equation}
 from which we can deduce~\eqref{eq:magic} by dividing by $t_{k+1}^2 - t_{k+1}$.

We plug in the estimate~\eqref{eq:magic} in~\eqref{eq:recursion-assumption-mu} and get for every $k \geq 2$
 \begin{equation}
   \begin{aligned}
     \mu_{k} &= \mu_1\frac{\prod_{j=1}^{k-1} t_j}{\prod_{j=1}^{k-1}(t_{j+1} - 1)} \frac{1}{t_{k}} \\
             &\le \mu_1 \prod_{j=1}^{k-1} \left( 1 + \frac{1}{t_{j+1}(t_{j+1} - 1)} \right) \frac{1}{t_{k}} \le \mu_1 \prod_{j=1}^{k-1}\left( 1 + \frac{4}{(j+2)j} \right) \frac{1}{t_{k}} \\
             &\le \mu_1 \prod_{j=1}^{k-1} \left( 1 + \frac{4}{j^2}\right) \frac{1}{t_{k}} \le \mu_1 \exp\left(\frac{\pi^2 4}{6}\right) \frac{1}{t_{k}} = b \lVert K \rVert^2\exp\left(\frac{\pi^2 4}{6}\right) \frac{1}{t_{k}}.
   \end{aligned}
 \end{equation}
 With the above inequalities we can to deduce the claimed convergence rates. First note that from Theorem~\ref{thm:smoothing_accelerated_1k} we have
  \begin{equation}
    F(x_{N}) - F(x^*) \le \frac{\lVert x_{0} - x^*\rVert^2 }{2\gamma_{N} t_{N}^{2}} + \mu_{N}\frac{L_{g}^2}{2} \quad \forall N \geq 1.
  \end{equation}
  Now, in order to obtain the desired conclusion, we used the above estimates and deduce for every $N \geq 1$
  \begin{equation}
    \begin{aligned}
      \frac{\lVert x_{0} - x^*\rVert^2 }{2\gamma_{N} t_{N}^{2}} + \mu_{N}\frac{L_{g}^2}{2} &\le 
      \frac{\lVert x_{0} - x^*\rVert^2 }{2 b t_{N}} + \frac{bL_{g}^2\lVert K \rVert^2}{2 t_N} \exp\left(\frac{4 \pi^2}{6}\right) \\
      &\le \frac{\lVert x_{0} - x^*\rVert^2 }{b (N+1)} + \frac{bL_{g}^2\lVert K \rVert^2}{(N+1)} \exp\left(\frac{4 \pi^2}{6}\right),
    \end{aligned}
  \end{equation}
  where we used that
  \begin{equation}
    \gamma_N t_N = \frac{\mu_N t_N}{\lVert K \rVert^2} \ge b.
  \end{equation}
\end{proof}

\begin{remark}
  Consider the choice (see~\cite{nesterov83})
  \begin{equation}
    t_1 = 1, \quad t_{k+1} = \frac{1 + \sqrt{1 + 4t_{k}^2}}{2} \quad \forall k \ge 1
  \end{equation}
and 
$$\mu_1 = b  \lVert K \rVert^2, \ \text{for} \ b >0.$$
Since
  \begin{equation}
    t_{k}^2 = t_{k+1}^2 - t_{k+1} \quad \forall k \geq 1,
  \end{equation}
we see that in this setting we have to choose
  \begin{equation}
    \mu_{k} = b \lVert K \rVert^2 \ \text{and} \ \gamma_k = b \quad \forall k \ge 1.
  \end{equation}
  Thus, the sequence of optimal function values ${(F(x_N))}_{N \geq 1}$ approaches a $b \lVert K \rVert^2 \frac{L_g}{2}$-approximation of the optimal objective value $F(x^*)$ with a convergence rate of $\mathcal{O}(\frac{1}{N^2})$, i.e.\
  \begin{equation}
    F(x_{N}) - F(x^*) \le 2\frac{\lVert x_{0} - x^*\rVert^2}{b {(N+1)}^2} + b\frac{\lVert K \rVert^2L_{g}^2}{2} \quad \forall N \ge 1.
  \end{equation}
\end{remark}

\section{Stochastic Method}%
\label{sec:var-smoothing-stochastic}

\begin{problem}%
  \label{prob:pd-lipschitz-stochastic}
  The problem is the same as in the deterministic case
  \begin{equation}
    \min_{x \in \H{}} f(x) + g(Kx)
  \end{equation}
  other than the fact that at each iteration we are only given a stochastic estimator of the quantity
  \begin{equation}
    \nabla (\sm{k}g \circ K)(\cdot) = K^* \prox{\frac{1}{\mu_{k}}g^*}{\frac{1}{\mu_{k}}K\cdot} \quad \forall k \geq 1.
  \end{equation}
\end{problem}
For the stochastic quantities arising in this section we will use the following notation. For every $k \geq 0$, we denote by $\sigma(x_0, \dots, x_k)$ the smallest $\sigma$-algebra generated by the family of random variables $\{x_0, \dots, x_k\}$ and by $\E_k(\cdot) := \E(\cdot | \sigma(x_0, \dots, x_k))$ the conditional expectation with respect to this $\sigma$-algebra.

\begin{algo}[stochastic Variable Accelerated SmooThing (sVAST)]%
  \label{alg:variable_smoothing_stochastic}
  Let $y_0 = x_0 \in \H, {(\mu_{k})}_{k \geq 1}$ a sequence of positive and nonincreasing real numbers, and ${(t_{k})}_{k \geq 1}$ a sequence of real numbers with $t_1=1$ and $t_k \geq 1$ for every $k \geq 2$. Consider the following iterative scheme
  \begin{equation*}
    (\forall k \geq 1) \quad 
    \left\lfloor \begin{array}{l}
      L_{k} = \frac{\lVert K \rVert^2}{\mu_{k}} \\
      \gamma_{k} = \frac{1}{L_{k}} \\
      x_{k} = \prox{\gamma_{k}f}{y_{k-1} - \gamma_{k}\xi_{k-1}}\\
      y_{k} = x_{k} + \frac{t_{k}-1}{t_{k+1}}(x_{k} - x_{k-1}),
    \end{array}\right.
  \end{equation*}
where we make the standard assumptions about our gradient estimator of being unbiased, i.e.\
  \begin{equation}
    \E_{k} (\xi_{k}) = \nabla (\sm{k+1}g \circ K)(y_{k}),
  \end{equation}
  and having bounded variance
  \begin{equation}
    \E_{k} \left(\lVert \xi_{k} - \nabla (\sm{k+1}g \circ K)(y_{k}) \rVert^2 \right) \le \sigma^2
  \end{equation}
  for every $k\ge0$.
\end{algo}
Note that we use the same notations as in the deterministic case
\begin{equation*}
  u_{k} := x_{k-1} + t_{k}(x_{k} - x_{k-1}) \ \mbox{and} \
  F^{k}(\cdot) := f + \sm{k}g\circ K \quad \forall k \geq 1.
\end{equation*}

\begin{lemma}%
  \label{lem:lipschitz-envelope}
  Let $g: \H \to \R$ be a convex and $L_g$-Lipschitz continuous function. Then its Moreau envelope $\sm{}g$ satisfies for every $x,y \in \H$
  \begin{equation}
    \lvert \sm{}g(x) - \sm{}g(y) \rvert \le L_{g}\lVert x-y \rVert + \frac{\mu L_{g}^2}{2}.
  \end{equation}
\end{lemma}
\begin{proof} 
  From Lemma~\ref{lem:smoothing-and-no-smoothing-lipschitz-constant} we deduce for every $x,y \in \H$
  \begin{equation}
    \sm{}g(x) - \sm{}g(y)  \le g(x) - g(y) + \frac{\mu L_{g}^2}{2} \le | g(x) - g(y) | + \frac{\mu L_{g}^2}{2}
  \end{equation}
  and, analogously,
  \begin{equation}
    \sm{}g(y) - \sm{}g(x)  \le g(y) - g(x) + \frac{\mu L_{g}^2}{2} \le | g(x) - g(y) | + \frac{\mu L_{g}^2}{2}.
  \end{equation}
  The statement of the lemma follows.
\end{proof}

\begin{lemma}%
  \label{lem:descent-comb-stoch}
The following statement holds for every $z \in \H$ and every $k \ge 0$
  \begin{equation}
    \label{eq:descent-comb}
    \begin{aligned}
      \E_{k} \left(F^{k+1}(x_{k+1}) + \frac{\lVert x_{k+1} - z \rVert^2}{2 \gamma_{k+1}}\right) \le & \ F^{k+1}(z) + \frac{\lVert z - y_{k} \rVert^2}{2 \gamma_{k+1}} \\
     & \ + \gamma_{k+1}\left(\sigma^2 + \frac{\lVert K \rVert^2}{2}\right)+ \mu_{k+1}\frac{L_{g}^2}{2}.
    \end{aligned}
  \end{equation}
\end{lemma}
\begin{proof}
  Here we have to proceed a little bit different from Lemma~\ref{lem:double-helpful-deterministic}.
  Namely, we have to treat the gradient step and the proximal step differently.
  For this purpose we define the auxiliary variable
  \begin{equation}
    z_{k} := y_{k-1} - \gamma_{k} \xi_{k-1} \quad \forall k\ge 1.
  \end{equation}
 Let be $k \geq 1$ fixed. From the gradient step we get
  \begin{equation}
    \begin{aligned}
      \lVert z - z_{k} \rVert^2 &= \lVert y_{k-1} - \gamma_{k}\xi_{k-1} - z \rVert^2 \\
      &= \lVert y_{k-1} - z \rVert^2 - 2\gamma_{k} \left\langle \xi_{k-1}, y_{k-1} - z \right\rangle + \gamma^2_{k}\lVert \xi_{k-1} \rVert^2.
    \end{aligned}
  \end{equation}
  Taking the conditional expectation gives
  \begin{equation}
    \begin{aligned}
      \E_{k-1}\left(\lVert z - z_{k} \rVert^2\right) \!= \lVert y_{k-1} - z \rVert^2 \! - 2\gamma_{k} \left\langle \nabla (\sm{k}g \circ K)(y_{k-1}), y_{k-1} - z \right\rangle \! + \gamma^2_{k}\E_{k-1} \left(\lVert \xi_{k-1} \rVert^2\right).
    \end{aligned}
  \end{equation}
  Using the gradient inequality we deduce
  \begin{equation}
    \begin{aligned}
      \E_{k-1}\left(\lVert z - z_{k} \rVert^2\right) \le & \  \lVert y_{k-1} - z \rVert^2 - 2\gamma_{k}((\sm{k}g \circ K)(y_{k-1}) - (\sm{k}g \circ K)(z))\\
& \ + \gamma^2_{k}\E_{k-1} \left(\lVert \xi_{k-1} \rVert^2\right)
    \end{aligned}
  \end{equation}
  and therefore
  \begin{equation}
    \label{eq:pure-grad-step}
    \begin{aligned}
      \gamma_{k}(\sm{k}g \circ K)(y_{k-1}) + \frac12\E_{k-1}\left(\lVert z - z_{k} \rVert^2\right) \le & \ \frac12\lVert y_{k-1} - z \rVert^2 + \gamma_{k}(\sm{k}g \circ K)(z) \\
& \ + \frac{\gamma^2_{k}}{2}\E_{k-1}\left(\lVert \xi_{k-1} \rVert^2\right).
    \end{aligned}
  \end{equation}
  Also from the smoothness of $(\sm{k}g \circ K)$ we deduce via the Descent Lemma that
  \begin{equation}
    \label{eq:descent-lemma}
    \begin{aligned}
      \sm{k}g (Kz_{k}) \le \sm{k}g (Ky_{k-1}) + \left\langle \nabla (\sm{k}g \circ K)(y_{k-1}), z_{k} - y_{k-1} \right\rangle + \frac{L_{k}}{2} \lVert z_{k} - y_{k-1} \rVert^2.
    \end{aligned}
  \end{equation}
  Plugging in the definition of $z_{k}$ and using the fact that $L_{k} = \frac{1}{\gamma_{k}}$ we get
  \begin{equation}
    \begin{aligned}
      \sm{k}g (Kz_{k}) \le \sm{k}g (Ky_{k-1}) - \gamma_{k} \left\langle \nabla (\sm{k}g \circ K)(y_{k-1}), \xi_{k-1} \right\rangle + \frac{\gamma_{k}}{2} \lVert \xi_{k-1} \rVert^2.
    \end{aligned}
  \end{equation}
  Now we take the conditional expectation to deduce that
  \begin{equation}
    \label{eq:descent-lemma-finished}
    \begin{aligned}
      \E_{k-1}(\sm{k}g (Kz_{k})) \le \sm{k}g (Ky_{k-1}) - \gamma_{k} \left\lVert \nabla (\sm{k}g \circ K)(y_{k-1}) \right\rVert^2 + \frac{\gamma_{k}}{2} \E_{k-1}\left(\lVert \xi_{k-1} \rVert^2\right).
    \end{aligned}
  \end{equation}
  Multiplying~\eqref{eq:descent-lemma-finished} by $\gamma_{k}$ and adding it to~\eqref{eq:pure-grad-step} gives
  \begin{equation}
    \begin{aligned}
      \gamma_{k} \E_{k-1} \left(\sm{k}g (Kz_{k})\right) + \frac12\E_{k-1}\left(\lVert z - z_{k} \rVert^2\right) & \le \\
      \gamma_{k}\sm{k}g (Kz) + \frac12\lVert y_{k-1} - z \rVert^2 - \gamma_{k}^2 \left\lVert \nabla (\sm{k}g \circ K)(y_{k-1}) \right\rVert^2 + \gamma_{k}^2 \E_{k-1}\left(\lVert \xi_{k-1} \rVert^2\right). &
    \end{aligned}
  \end{equation}
  Now we use the assumption about the bounded variance to deduce that
  \begin{equation}
    \label{eq:grad-step-finished-stoch}
    \begin{aligned}
      \gamma_{k}\E_{k-1}\left(\sm{k}g (Kz_{k})\right) + \frac12\E_{k-1}\left(\lVert z - z_{k} \rVert^2\right) \le 
      \gamma_{k}\sm{k}g (Kz) + \frac12\lVert y_{k-1} - z \rVert^2 + \gamma_{k}^2 \sigma^2.
    \end{aligned}
  \end{equation}
  Next up for the proximal step we deduce
  \begin{equation}
    \label{eq:prox-step-stoch}
    f(x_{k}) + \frac{1}{2 \gamma_{k}}\lVert x_{k} - z_{k} \rVert^2 + \frac{1}{2 \gamma_{k}}\lVert x_{k} - z \rVert^2 \le f(z) + \frac{1}{2 \gamma_{k}} \lVert z - z_{k} \rVert^2.
  \end{equation}
Taking the conditional expectation and combining~\eqref{eq:grad-step-finished-stoch} and~\eqref{eq:prox-step-stoch} we get
  \begin{equation}
    \begin{aligned}
      \E_{k-1} \left(\gamma_{k}(\sm{k}g(Kz_{k}) + f(x_{k})) + \frac12 \lVert x_{k} - z_{k} \rVert^2 +\frac12 \lVert x_{k} - z \rVert^2 \right) & \le \\
      \gamma_{k}F^{k}(z) +\frac12\lVert y_{k-1} - z \rVert^2 + \gamma^2_{k}\sigma^2. &
    \end{aligned}
  \end{equation}
 From here, using now Lemma~\ref{lem:lipschitz-envelope}, we get that
  \begin{equation}
    \begin{aligned}
      \E_{k-1} \left(\gamma_{k}F^{k}(x_{k}) - \gamma_{k}\lVert K \rVert \lVert x_{k} - z_{k} \rVert - \mu_{k}\gamma_{k}\frac{L_{g}^2}{2} + \frac12 \lVert x_{k} - z_{k} \rVert^2 +\frac12 \lVert x_{k} - z \rVert^2 \right)& \le \\
      \gamma_{k}F^{k}(z) +\frac12\lVert y_{k-1} - z \rVert^2 + \gamma^2_{k}\sigma^2. &
    \end{aligned}
  \end{equation}
  Now we use
  \begin{equation}
      - \frac12\gamma_{k}^2 \lVert K \rVert^2 \le \frac12\lVert x_{k} - z_{k} \rVert^2 - \gamma_{k}\lVert K \rVert \lVert x_{k} - z_{k} \rVert   
  \end{equation}
  to obtain that
  \begin{equation}
    \begin{aligned}
      \E_{k-1} \left(\gamma_{k}F^{k}(x_{k})   + \frac12 \lVert x_{k} - z \rVert^2 \right)
     & \le \\
 \gamma_{k}F^{k}(z) +\frac12\lVert y_{k-1} - z \rVert^2 + \gamma^2_{k}\sigma^2 + \mu_{k}\gamma_{k}\frac{L_{g}^2}{2} + \frac12 \gamma_{k}^2 \lVert K \rVert^2. &
    \end{aligned}
  \end{equation}
\end{proof} 

\begin{lemma}%
  \label{lem:initial-stochastic}
  Let $x^*$ be an optimal solution of Problem~\ref{prob:pd-lipschitz-stochastic}. Then it holds
  \begin{equation}
    \begin{aligned}
      \E \left(\gamma_{1}(F^{1}(x_{1}) - F^{1}(x^*))\right) + \frac12 \lVert u_{1} - x^* \rVert^2
      \le \frac12\lVert x_{0} - x^* \rVert^2 + \gamma^2_{1}\sigma^2 + \mu_{1}\gamma_{1}\frac{L_{g}^2}{2} + \frac12 \gamma_{1}^2 \lVert K \rVert^2.
    \end{aligned}
  \end{equation}
\end{lemma}
\begin{proof}
  Applying the previous lemma with $k=0$ and $z = x^*$, we get that 
  \begin{equation}
    \begin{aligned}
      \E \left(\gamma_{1}F^{1}(x_{1}) + \frac12 \lVert x_{1} - x^* \rVert^2\right)
      \! \le \! \gamma_{1}F^{1}(x^*) +\frac12\lVert y_{0} - x^* \rVert^2 + \gamma^2_{1}\sigma^2 + \mu_{1}\gamma_{1}\frac{L_{g}^2}{2} + \frac12 \gamma_{1}^2 \lVert K \rVert^2.
    \end{aligned}
  \end{equation}
  Therefore, using the fact that $y_{0} = x_{0}$ and $u_{1} = x_{1}$,
  \begin{equation}
    \begin{aligned}
      \E \left(\gamma_{1}(F^{1}(x_{1}) - F^{1}(x^*)) + \frac12 \lVert u_{1} - x^* \rVert^2\right)
      \le \frac12\lVert x_{0} - x^* \rVert^2 + \gamma^2_{1}\sigma^2 + \mu_{1}\gamma_{1}\frac{L_{g}^2}{2} + \frac12 \gamma_{1}^2 \lVert K \rVert^2,
    \end{aligned}
  \end{equation}
  which finishes the proof.
\end{proof}

\begin{theorem}%
  \label{thm:var-smoothing-stoch}
  Consider the setup of Problem~\ref{prob:pd-lipschitz-stochastic} and let ${(x_{k})}_{k \ge 0}$ and ${(y_{k})}_{k \ge 0}$ denote the sequences generated by Algorithm~\ref{alg:variable_smoothing_stochastic}. Assume that for all $k\ge 1$
  \begin{equation}
    \rho_{k+1} := t_{k}^2 - t_{k+1}^2 + t_{k+1} \ge 0.
  \end{equation}
  Then, for every optimal solution $x^*$ of Problem~\ref{prob:pd-lipschitz-stochastic}, it holds 
  \begin{align*}
      \E \left(F(x_{N}) - F(x^*)\right) \le & \ \frac{1}{\gamma_{N} t_{N}^{2}} \frac12\lVert x_{0} - x^*\rVert^2 + 
      \frac{1}{\gamma_{N} t_{N}^{2}}\frac{L_g^2}{2\lVert K \rVert^2}\sum_{k=1}^{N} \mu_{k}^2(t_{k} +
      \rho_{k}) \\
& \ + \frac{1}{\gamma_{N} t_{N}^{2}} \left(\frac{\sigma^2}{\lVert K \rVert^4} +
      \frac{L_g^2+1}{2\lVert K \rVert^2}\right) \sum_{k=1}^{N} t_{k}^2\mu_{k}^2 \quad \forall N \geq 1.
  \end{align*}
\end{theorem}
\begin{proof}[Proof of Theorem~\ref{thm:var-smoothing-stoch}]
 Let be $k \geq 0$ fixed. Lemma~\ref{lem:descent-comb-stoch} for $z := \left( 1- \frac{1}{t_{k+1}} \right)x_{k} + \frac{1}{t_{k+1}} x^*$ gives
  \begin{align*}
    \E_{k} \left(F^{k+1}(x_{k+1}) + \frac{1}{2 \gamma_{k+1}}\left\lVert \frac{1}{t_{k+1}} u_{k+1} - \frac{1}{t_{k+1}} x^* \right\rVert^2 \right) & \le \\
         F^{k+1}\left(\left( 1- \frac{1}{t_{k+1}} \right)x_{k} + \frac{1}{t_{k+1}} x^*\right) + \frac{1}{2 \gamma_{k+1}}\left\lVert \frac{1}{t_{k+1}} x^* - \frac{1}{t_{k+1}} u_{k} \right\rVert^2  & \\ 
         +\gamma_{k+1}\left(\sigma^2 + \frac{\lVert K \rVert^2}{2}\right) + \mu_{k+1}\frac{L_{g}^2}{2}. &
  \end{align*}
From here and from the convexity of $F^{k+1}$ follows
  \begin{align*}
    \E_{k} \left(F^{k+1}(x_{k+1}) - F^{k+1}(x^*) \right) - \left(1 - \frac{1}{t_{k+1}} \right) (F^{k+1}(x_{k}) - F^{k+1}(x^*)) & \le \\
    \frac{\lVert u_{k} - x^* \rVert^2}{2 \gamma_{k+1} t_{k+1}^2} - \E_{k}\left(\frac{\lVert u_{k+1} - x^* \rVert^2}{2 \gamma_{k+1} t_{k+1}^2}\right) + \gamma_{k+1}\left(\sigma^2 + \frac{\lVert K \rVert^2}{2}\right) + \mu_{k+1}\frac{L_{g}^2}{2}.&
  \end{align*}
  Now, by multiplying both sides with by $t_{k+1}^2$, we deduce
  \begin{equation}
    \begin{aligned}
    \label{eq:lala-stoch}
    \E_{k} \left(t_{k+1}^2 (F^{k+1}(x_{k+1}) - F^{k+1}(x^*)) \right) + (t_{k+1} - t_{k+1}^2)(F^{k+1}(x_{k}) - F^{k+1}(x^*)) & \le\\
    \frac{1}{2 \gamma_{k+1}} \left (\lVert u_{k} - x^* \rVert^2 - \E_{k}\left(\lVert u_{k+1} - x^* \rVert^2 \right) \right)+t_{k+1}^2\gamma_{k+1}\left(\sigma^2 + \frac{\lVert K \rVert^2}{2}\right) + t_{k+1}^2\mu_{k+1}\frac{L_{g}^2}{2}. &
  \end{aligned}
    \end{equation}
  Next, by adding $t_{k}^2(F^{k+1}(x_{k}) - F^{k+1}(x^*))$ on both sides of~\eqref{eq:lala-stoch}, gives
  \begin{align*}
    \E_{k} \left( t_{k+1}^2 (F^{k+1}(x_{k+1}) - F^{k+1}(x^*))\right) + \rho_{k+1}(F^{k+1}(x_{k}) - F^{k+1}(x^*))  & \le \\
    t_{k}^2(F^{k+1}(x_{k}) - F^{k+1}(x^*)) + \frac{1}{2 \gamma_{k+1}} \left(\lVert u_{k} - x^* \rVert^2 - \E_{k} \left(\lVert u_{k+1} - x^* \rVert^2 \right) \right) & \\
      + t_{k+1}^2\gamma_{k+1}\left(\sigma^2 + \frac{\lVert K \rVert^2}{2}\right) + t_{k+1}^2\mu_{k+1}\frac{L_{g}^2}{2}. &
  \end{align*}
  Utilizing~\eqref{eq:est-diff-smoothing} together with the assumption that ${(\mu_{k})}_{k \geq 1}$ is nonincreasing leads to
  \begin{align*}
    \E_{k} \left(t_{k+1}^2 (F^{k+1}(x_{k+1}) - F^{k+1}(x^*)) \right) + \rho_{k+1}(F^{k+1}(x_{k}) - F^{k+1}(x^*)) & \le \\
    t_{k}^2(F^{k}(x_{k}) - F^{k}(x^*)) + \frac{1}{2\gamma_{k+1}} \left (\lVert u_{k} - x^* \rVert^2 - \E_{k} \left(\lVert u_{k+1} - x^* \rVert^2 \right) \right) + t_{k}^2(\mu_{k} - \mu_{k+1})\frac{L_{g}^2}{2}  & \\ 
      + t_{k+1}^2\gamma_{k+1}\left(\sigma^2 + \frac{\lVert K \rVert^2}{2}\right) + t_{k+1}^2\mu_{k+1}\frac{L_{g}^2}{2}. &
  \end{align*}
  Now, using that $t_{k}^2 \ge t_{k+1}^2 - t_{k+1}$, we get
  \begin{align*}
    \E_{k}\left( t_{k+1}^2 (F^{k+1}(x_{k+1}) - F^{k+1}(x^*)) \right) + \rho_{k+1}(F^{k+1}(x_{k}) - F^{k+1}(x^*)) & \le \\
    t_{k}^2(F^{k}(x_{k}) - F^{k}(x^*)) + \frac{1}{2\gamma_{k+1}} (\lVert u_{k} - x^* \rVert^2 - \E_{k}\left(\lVert u_{k+1} - x^* \rVert^2) \right)  & \\
      + t_{k}^2\mu_{k} \frac{L_{g}^2}{2} - t_{k+1}^2 \mu_{k+1}\frac{L_{g}^2}{2} + t_{k+1} \mu_{k+1}\frac{L_{g}^2}{2} & \\
      +t_{k+1}^2\gamma_{k+1}\left(\sigma^2 + \frac{\lVert K \rVert^2}{2}\right) + t_{k+1}^2\mu_{k+1}\frac{L_{g}^2}{2}. &
  \end{align*}
  Multiplying both sides with $\gamma_{k+1}$ and putting all terms on the correct sides yields
  \begin{equation}
    \label{eq:need-to-add-something-so-rho-term-is-positive-stoch}
\begin{aligned}
    \E_{k}\Bigg(\gamma_{k+1}t_{k+1}^2 \left(F^{k+1}(x_{k+1}) - F^{k+1}(x^*) + \mu_{k+1}\frac{L_{g}^2}{2}\right) + \frac12\lVert u_{k+1} - x^* \rVert^2 \Bigg) + & \\ 
      \gamma_{k+1}\rho_{k+1}(F^{k+1}(x_{k}) - F^{k+1}(x^*)) & \le \\
      \gamma_{k+1}t_{k}^2\left(F^{k}(x_{k}) - F^{k}(x^*) + \mu_{k}\frac{L_{g}^2}{2}\right) + \frac12\lVert u_{k} - x^* \rVert^2  + & \\
      + \gamma_{k+1} t_{k+1} \mu_{k+1}\frac{L_{g}^2}{2} 
      +t_{k+1}^2\gamma_{k+1}^2\left(\sigma^2 + \frac{\lVert K \rVert^2}{2}\right) + \gamma_{k+1}t_{k+1}^2\mu_{k+1}\frac{L_{g}^2}{2}. &
  \end{aligned}
\end{equation}
  At this point we would like to discard the term $\gamma_{k+1} \rho_{k+1}(F^{k+1}(x_{k}) - F^{k+1}(x^*))$ which we currently cannot as the positivity of $F^{k+1}(x_{k}) - F^{k+1}(x^*)$ is not ensured.
  So we add $\gamma_{k+1}\rho_{k+1}\mu_{k+1} \frac{L_{g}^2}{2}$ on both sides of~\eqref{eq:need-to-add-something-so-rho-term-is-positive-stoch} and get
  \begin{equation}
    \label{eq:can-now-be-discarded-stoch}
\begin{aligned}
    \E_k\Bigg(\gamma_{k+1}t_{k+1}^2 \left(F^{k+1}(x_{k+1}) - F^{k+1}(x^*) + \mu_{k+1}\frac{L_{g}^2}{2}\right) + \frac12\lVert u_{k+1} - x^* \rVert^2  \Bigg) + & \\ 
      \gamma_{k+1}\rho_{k+1}\left(F^{k+1}(x_{k}) - F^{k+1}(x^*) + \mu_{k+1}\frac{L_{g}^2}{2}\right) & \le \\
      \gamma_{k+1}t_{k}^2\left(F^{k}(x_{k}) - F^{k}(x^*) + \mu_{k}\frac{L_{g}^2}{2} \right) + \frac12\lVert u_{k} - x^* \rVert^2 + & \\
      + \gamma_{k+1} \mu_{k+1}\frac{L_{g}^2}{2}(t_{k+1}+\rho_{k+1})
      +t_{k+1}^2\gamma_{k+1}^2\left(\sigma^2 + \frac{\lVert K \rVert^2}{2}\right) + \gamma_{k+1}t_{k+1}^2\mu_{k+1}\frac{L_{g}^2}{2}. &
  \end{aligned}
\end{equation}
  Using again~\eqref{eq:est-diff-smoothing} to deduce that 
  \begin{equation}
    \gamma_{k+1}\rho_{k+1}\left(F^{k+1}(x_{k}) - F^{k+1}(x^*) + \mu_{k+1}\frac{L_{g}^2}{2}\right) \ge \gamma_{k+1}\rho_{k+1}(F(x_{k}) - F(x^*)) \ge 0
  \end{equation}
  we can now discard said term from~\eqref{eq:can-now-be-discarded-stoch}, giving
  \begin{equation}
    \label{eq:one-more-before-sum-it-up-stoch}
\begin{aligned}
    \E_{k}\left(\gamma_{k+1}t_{k+1}^2 \left(F^{k+1}(x_{k+1}) - F^{k+1}(x^*) + \mu_{k+1}\frac{L_{g}^2}{2}\right) + \frac{1}{2}\lVert u_{k+1} - x^* \rVert^2 \right) & \le \\ 
    \gamma_{k+1}t_{k}^2\left(F^{k}(x_{k}) - F^{k}(x^*) + \mu_{k}\frac{L_{g}^2}{2}\right) + \frac{1}{2}\lVert u_{k} - x^* \rVert^2 & \\
      + \gamma_{k+1} \mu_{k+1}\frac{L_{g}^2}{2}(t_{k+1}+\rho_{k+1})
      +t_{k+1}^2\gamma_{k+1}^2\left(\sigma^2 + \frac{\lVert K \rVert^2}{2}\right) + \gamma_{k+1}t_{k+1}^2\mu_{k+1}\frac{L_{g}^2}{2}. &
   \end{aligned}
\end{equation}
  Last but not least we use the that $F^{k}(x_{k}) - F^{k}(x^*) + \mu_{k}\frac{L_{g}^2}{2} \ge F(x_{k})-F(x^*)\ge 0$ and $\gamma_{k+1} \le \gamma_{k}$ to follow that
  \begin{equation}
    \label{eq:decrease-gamma-stoch}
    \gamma_{k+1}t_{k}^2\left(F^{k}(x_{k}) - F^{k}(x^*) + \mu_{k}\frac{L_{g}^2}{2}\right) \le 
    \gamma_{k}t_{k}^2\left(F^{k}(x_{k}) - F^{k}(x^*) + \mu_{k}\frac{L_{g}^2}{2}\right).
  \end{equation}
  Combining~\eqref{eq:one-more-before-sum-it-up-stoch} and~\eqref{eq:decrease-gamma-stoch} yields
  \begin{equation}
    \label{eq:sum-it-up-stoch}
\begin{aligned}
    \E_{k}\left(\gamma_{k+1}t_{k+1}^2 \left(F^{k+1}(x_{k+1}) - F^{k+1}(x^*) + \mu_{k+1}\frac{L_{g}^2}{2}\right) + \frac12\lVert u_{k+1} - x^* \rVert^2 \right) & \le \\ 
    \gamma_{k}t_{k}^2\left(F^{k}(x_{k}) - F^{k}(x^*) + \mu_{k}\frac{L_{g}^2}{2}\right) + \frac12\lVert u_{k} - x^* \rVert^2  & \\
      + \gamma_{k+1} \mu_{k+1}\frac{L_{g}^2}{2}(t_{k+1}+\rho_{k+1})
      +t_{k+1}^2\gamma_{k+1}^2\left(\sigma^2 + \frac{\lVert K \rVert^2}{2}\right) + \gamma_{k+1}t_{k+1}^2\mu_{k+1}\frac{L_{g}^2}{2}. &
   \end{aligned}
\end{equation}
Let be $N \geq 2$. We take the expected value on both sides~\eqref{eq:sum-it-up-stoch} and sum from $k=1$ to $N-1$. Getting rid of the non-negative terms $\lVert u_{N} - x^* \rVert^2$ gives
\begin{equation}
 \begin{aligned} 
      \E\left(\gamma_{N} t_{N}^{2}\left(F^{N}(x_{N}) - F^{N}(x^*) + \mu_{N}\frac{L_{g}^2}{2}\right) \right) & \le \\
      \E\left(\gamma_{1}\left(F^{1}(x_{1}) - F^{1}(x^*) + \mu_{1}\frac{L_g^2}{2}\right)\right) + \frac12\lVert u_{1}- x^* \rVert^2 + \sum_{k=2}^{N} \gamma_{k} \mu_{k}\frac{L_{g}}{2}(t_{k} + \rho_{k}) & \\
      + \sum_{k=2}^{N} t_{k}^2\gamma_{k}^2\left(\sigma^2 + \frac{\lVert K \rVert^2}{2}\right) + \gamma_{k}t_{k}^2\mu_{k}\frac{L_{g}}{2}. &
 \end{aligned}
 \end{equation}

Since $t_1=1$, the above inequality holds also for $N=1$. Now, using Lemma~\ref{lem:initial-stochastic} we get that for every $N \geq 1$
  \begin{align*}
    \E\left(\gamma_{N} t_{N}^{2}\left(F^{N}(x_{N}) - F^{N}(x^*) + \mu_{N}\frac{L_{g}^2}{2} \right) \right) \le & \ \frac12\lVert x_{0} - x^*\rVert^2 + \sum_{k=1}^{N} \gamma_{k} \mu_{k}\frac{L_g^2}{2}(t_{k} + \rho_{k}) \\ & \ + \sum_{k=1}^{N} t_{k}^2\gamma_{k}^2\left(\sigma^2 + \frac{\lVert K \rVert^2}{2}\right) + \gamma_{k}t_{k}^2\mu_{k}\frac{L_g^2}{2}.
  \end{align*}
  From Lemma~\ref{lem:smoothing-and-no-smoothing-lipschitz-constant} we follow that 
  \begin{equation*}
    \begin{aligned}
      &\gamma_{N} t_{N}^{2} \left( F(x_{N}) - F(x^*) \right) \le \gamma_{N} t_{N}^{2} \left( F^{N}(x_{N}) - F^{N}(x^*) + \mu_{N}\frac{L_g^2}{2} \right),
    \end{aligned}
  \end{equation*}
  therefore, for every $N \geq 1$
  \begin{align*}
    \E\left(\gamma_{N} t_{N}^{2}\left(F^{N}(x_{N}) - F^{N}(x^*) \right) \right) \le & \ \frac12\lVert x_{0} - x^*\rVert^2 + \sum_{k=1}^{N} \gamma_{k} \mu_{k}\frac{L_g^2}{2}(t_{k} + \rho_{k}) \\ & \ + \sum_{k=1}^{N} t_{k}^2\gamma_{k}^2\left(\sigma^2 + \frac{\lVert K \rVert^2}{2}\right) + \gamma_{k}t_{k}^2\mu_{k}\frac{L_g^2}{2}.
  \end{align*}
  By using the fact that $\gamma_{k} = \frac{1}{L_{k}}= \frac{\mu_{k}}{\lVert K \rVert^2}$ for every $k \ge 1$ gives
  \begin{align*}
    \E\left(\gamma_{N} t_{N}^{2}(F(x_{N}) - F(x^*)) \right) \le & \ \frac12\lVert x_{0} - x^*\rVert^2 +
      \frac{L_g^2}{2\lVert K \rVert^2}\sum_{k=1}^{N} \mu_{k}^2(t_{k} + \rho_{k}) \\
& \ + \left(\frac{\sigma^2}{\lVert K \rVert^4} + \frac{L_g^2 + 1}{2 \lVert K \rVert^2}\right) \sum_{k=1}^{N} t_{k}^2\mu_{k}^2 \quad \forall N \geq 1.
  \end{align*}
  Thus,
  \begin{align*}
    \E\left( F(x_{N}) - F(x^*) \right) \le & \ \frac{1}{\gamma_{N} t_{N}^{2}} \frac12\lVert x_{0} - x^*\rVert^2 +
      \frac{1}{\gamma_{N} t_{N}^{2}}\frac{L_g^2}{2\lVert K \rVert^2}\sum_{k=1}^{N} \mu_{k}^2(t_{k} +
      \rho_{k}) \\
& \ + \frac{1}{\gamma_{N} t_{N}^{2}} \left(\frac{\sigma^2}{\lVert K \rVert^4} +
      \frac{L_g^2 + 1}{2\lVert K \rVert^2}\right) \sum_{k=1}^{N} t_{k}^2\mu_{k}^2 \quad \forall N \geq 1.
  \end{align*}
\end{proof}

\begin{corollary}%
  \label{cor:param-choices-stochastic}
  Let
  \begin{equation*}
 t_1=1, \quad t_{k+1} = \frac{1 + \sqrt{1 + 4 t_{k}^2}}{2} \quad \forall k \geq 1,
  \end{equation*}
and, for $b >0$,
\begin{equation*}
   \mu_{k} = \frac{b}{k^{\frac32}} \lVert K \rVert^2, \ \text{and} \ \gamma_k = \frac{b}{k^{\frac32}} \quad \forall k \geq 1.
\end{equation*}
  Then,
  \begin{align*}
   \E\left( F(x_{N}) - F(x^*) \right) \le  & \  2\frac{\lVert x_0 - x^* \rVert^2}{b\sqrt{N}} + 
    L_{g}^2\lVert K \rVert^2 b^2 \frac{\pi^2}{6}\frac{1}{\sqrt{N}} \\
    & \ + 2b^2\left( 2\sigma^2 + L_g^2\lVert K \rVert^2 + \lVert K \rVert^2 \right) \frac{1+\log(N)}{\sqrt{N}} \quad \forall N \geq 1.
  \end{align*}
  Furthermore, we have that $F(x_N)$ converges almost surely to $F(x^*)$ as $N \rightarrow +\infty$.
\end{corollary}

\begin{proof}
  First we notice that the choice of $t_{k+1} = \frac{1 + \sqrt{1 + 4 t_{k}^2}}{2}$ fulfills that 
  \begin{equation*}
    \rho_{k+1} = t_{k}^2 - t_{k+1}^2 + t_{k+1} = 0 \quad \forall k \ge 1.
  \end{equation*}
  Now we derive the stated convergence result by first showing via induction that
  \begin{equation*}
   \frac{1}{k} \le \frac{1}{t_{k}}  \le \frac{2}{k} \quad \forall k \geq 1.
  \end{equation*}
Assuming that this holds for $k \geq 1$, we have that
  \begin{equation*}
    t_{k+1} = \frac{1+\sqrt{1+4t_{k}^2}}{2} \le \frac{1+\sqrt{1+4k^2}}{2} \le \frac{1+\sqrt{1+4k+4k^2}}{2} = k+1
  \end{equation*}
  and
  \begin{equation*}
    t_{k+1} =\frac{1+\sqrt{1+4t_{k}^2}}{2} \ge\frac{1+\sqrt{1+4{(\frac{k}{2})}^2}}{2} \ge \frac{1+\sqrt{k^2}}{2}\ge \frac{k+1}{2}.
  \end{equation*}
Furthermore, for every $N \geq 1$ we have that
  \begin{equation}
    \label{eq:estimate-sum}
    \begin{aligned}
      \frac{1}{\gamma_{N} t_{N}^{2}}\frac{L_g^2}{2\lVert K \rVert^2}\sum_{k=1}^{N} \mu_{k}^2(t_{k} + \rho_{k}) \le &
      \ \frac{2}{\sqrt{N}}\frac{L_g^2}{2\lVert K \rVert^4}\sum_{k=1}^{N} \frac{b^2 \lVert K \rVert^2}{k^{3}}k  = \frac{b^2 L_g^2 \lVert K \rVert^2}{\sqrt{N}}\sum_{k=1}^{N} k^{-2} \\
\le & \ \frac{b^2 L_g^2 \lVert K \rVert^2}{\sqrt{N}}\sum_{k=1}^{+\infty} k^{-2} = L_{g}^2b^2\lVert K \rVert^2 \frac{\pi^2}{6}\frac{1}{\sqrt{N}}.
    \end{aligned}
  \end{equation}
  The statement of the convergence rate in expectation follows now by plugging in our parameter choices into the statement of Theorem~\ref{thm:var-smoothing-stoch}, using the estimate~\eqref{eq:estimate-sum} and checking that
  \begin{equation}
    \sum_{k=1}^{N} t_{k}^2 \mu_{k}^2 \le b^2 \lVert K \rVert^2 \sum_{k=1}^{N} \frac{1}{k} \le b^2 \lVert K \rVert^2 (1+\log(N)) \quad \forall N \geq 1.
  \end{equation}
  The almost sure convergence of ${(F(x_N))}_{N \geq 1}$ can be deduced by looking at~\eqref{eq:sum-it-up-stoch} and dividing by $\gamma_{k+1}t_{k+1}^2$, which gives for every $k \geq 0$
\begin{equation}
\begin{aligned}
    & \E_{k}\left(F^{k+1}(x_{k+1}) - F^{k+1}(x^*) + \mu_{k+1}\frac{L_{g}^2}{2} + \frac{1}{\gamma_{k+1}t_{k+1}^2 }\lVert u_{k+1} - x^* \rVert^2 \right)  \le \\ 
    & F^{k}(x_{k}) - F^{k}(x^*) + \mu_{k}\frac{L_{g}^2}{2} + \frac{1}{\gamma_{k}t_{k}^2 }\lVert u_{k} - x^* \rVert^2  + \\
    &  \frac{1}{\gamma_{k+1}t_{k+1}^2 }\left(\gamma_{k+1} \mu_{k+1}\frac{L_{g}^2}{2}(t_{k+1}+\rho_{k+1})
    +t_{k+1}^2\gamma_{k+1}^2\left(\sigma^2 + \frac{\lVert K \rVert^2}{2}\right) + \gamma_{k+1}t_{k+1}^2\mu_{k+1}\frac{L_{g}^2}{2}\right).
     \end{aligned}
  \end{equation}
  Using the fact that $\gamma_{k+1}t_{k+1}^2 \ge \gamma_{k}t_{k}^2 $ we deduce that for every $k \geq 0$
  \begin{equation}
\begin{aligned}
   &  \E_{k}\left(F^{k+1}(x_{k+1}) - F^{k+1}(x^*) + \mu_{k+1}\frac{L_{g}^2}{2} + \frac{1}{\gamma_{k+1}t_{k+1}^2 }\lVert u_{k+1} - x^* \rVert^2 \right) \le \\ 
    & F^{k}(x_{k}) - F^{k}(x^*) + \mu_{k}\frac{L_{g}^2}{2} + \! \frac{1}{\gamma_{k}t_{k}^2 }\lVert u_{k} - x^* \rVert^2 + \! \frac{\mu_{k+1}}{t_{k+1}}\frac{L_{g}^2}{2}
  +\gamma_{k+1}\left(\sigma^2 + \frac{\lVert K \rVert^2}{2} + \frac{L_{g}^2}{2 \lVert K \rVert^2}\right).
  \end{aligned}
\end{equation}
  Plugging in our choice of parameters gives for every $k \geq 0$
  \begin{equation}
\begin{aligned}
    & \E_{k}\left(F^{k+1}(x_{k+1}) - F^{k+1}(x^*) + \mu_{k+1}\frac{L_{g}^2}{2} + \frac{1}{\gamma_{k+1}t_{k+1}^2 }\lVert u_{k+1} - x^* \rVert^2 \right) \le \\ 
    & F^{k}(x_{k}) - F^{k}(x^*) + \mu_{k}\frac{L_{g}^2}{2} + \frac{1}{\gamma_{k}t_{k}^2 }\lVert u_{k} - x^* \rVert^2
    +\frac{C}{k^{\frac32}},
  \end{aligned}
\end{equation}
where $C >0$.
 
Thus, by the famous Robbins-Siegmund Theorem (see~\cite[Theorem 1]{robbins-siegmund}) we get that ${(F^{k+1}(x_{k+1}) - F^{k+1}(x^*) + \mu_{k+1}\frac{L_{g}^2}{2})}_{k \geq 0}$ converges almost surely. In particular, from the convergence to $0$ in expectation we know that the almost sure limit must also be the constant zero.
\end{proof}

The formulation of the previous section can be used to deal e.g.\ with problems of the form
\begin{equation}
  \label{eq:finite-sum-concrete}
  \min_{x \in \H} f(x) + \sum_{i=1}^{m} g_{i}(K_{i}x)
\end{equation}
for $f:\H \to \overbar{\R}$ a proper, convex and lower semicontinuous function, $g_{i}:\G_{i} \to \R$ convex and $L_{g_i}$-Lipschitz continuous functions and $K_{i}: \H \to \G_{i}$ linear continuous operators for $i=1,\dots,m$.

Clearly one could consider
\begin{equation}
  \bm{K} := 
  \begin{cases}%
    \H \to \bigtimes_{i=1}^m \G_i \\
    x \mapsto \bigtimes_{i=1}^m K_i x
  \end{cases}
\end{equation}
and
\begin{equation}
  \bm{g} :=
  \begin{cases}
    \bigtimes_{i=1}^m \G_i \to \overbar{\R} \\
    \bigtimes_{i=1}^m y_i \mapsto \sum_{i=1}^m g_i(y_i).
  \end{cases}
\end{equation}
in order to reformulate the problem as
\begin{equation}
  \min_{x \in \H} f(x) + \bm{g}(\bm{K}x)
\end{equation}
and use Algorithm~\ref{alg:variable_smoothing_accelerated} together with the parameter choices described in Corollary~\ref{cor:param-choices-deterministic} on this. This results in the following algorithm.
\begin{algo}%
  Let $y_0 = x_0 \in \H, \mu_{1}=b \lVert \bm{K} \rVert$, for $b >0 $, and $t_1=1$. Consider the following iterative scheme
  \begin{equation}
    (\forall k \geq 1) \quad 
    \left\lfloor \begin{array}{l}
        \gamma_{k} = \frac{\sum_{i=1}^{m} \lVert K_i \rVert^2}{\mu_{k}}  \\
        x_{k} = \prox{\gamma_{k}f}{y_{k-1} - \gamma_{k} \sum_{i=1}^m K^*_i\prox{\frac{1}{\mu_{k}}g^*_i}{\frac{K_i y_{k-1}}{\mu_{k}}} } \\
        t_{k+1} = \sqrt{t_{k}^2 + 2t_{k}} \\
        y_{k} = x_{k} + \frac{t_{k}-1}{t_{k+1}}(x_{k} - x_{k-1}) \\
        \mu_{k+1} = \mu_{k} \frac{t_{k}^2}{t_{k+1}^2 - t_{k+1}}.
    \end{array}\right.
  \end{equation}
\end{algo}
 However, problem~\eqref{eq:finite-sum-concrete} also lends itself to be tackled via the stochastic version of our method, Algorithm~\ref{alg:variable_smoothing_stochastic}, by randomly choosing a subset of the summands. Together with the parameter choices described in Corollary~\ref{cor:param-choices-stochastic} which results in the following scheme.

\begin{algo}%
  Let $y_0 = x_0 \in \H, b >0$, and $t_1=1$. Consider the following iterative scheme
  \begin{equation}
    (\forall k \geq 1) \quad 
    \left\lfloor \begin{array}{l}
        \mu_{k} = b\sum_{i=1}^{m} \lVert K_i \rVert^2 k^{-\frac32} \\
        \gamma_{k} = b k^{-\frac32}  \\
        x_{k} = \prox{\gamma_{k} f}{y_{k-1} - \gamma_{k} \frac{\epsilon_{i,k}}{p_i} \sum_{i=1}^m K^*_i\prox{\frac{1}{\mu_{k}}g^*_i}{\frac{K_i y_{k-1}}{\mu_{k}}} } \\
      t_{k+1} = \frac{1 + \sqrt{1 + 4 t_{k}^2}}{2} \\
      y_{k} = x_{k} + \frac{t_{k}-1}{t_{k+1}}(x_{k} - x_{k-1}),
    \end{array}\right.
  \end{equation}
  with $\epsilon_{k} := (\epsilon_{1,k}, \epsilon_{2,k}, \dots, \epsilon_{m,k})$ a sequence of i.i.d., ${\{0,1\}}^m$ random variables and $p_i = \P[\epsilon_{i,1} = 1]$.
\end{algo}

\begin{remark}
  In theory Algorithm~\ref{alg:variable_smoothing_stochastic} could be used to treat more general stochastic problems than finite sums like~\eqref{eq:finite-sum-concrete}, but in this case it is not clear anymore how a gradient estimator can be found, so we do not discuss it here.
\end{remark}

\section{Numerical Examples}%
\label{sec:numerical_examples}

We will focus our numerical experiments on image processing problems. The examples are implemented in python using the operator discretization library (ODL)~\cite{odl}.
We define the discrete gradient operators $D_1$ and $D_2$ representing the discretized derivative in the first and second coordinate respectively, which we will need for the numerical examples.
Both map from $\R^{m \times n}$ to $\R^{m \times n}$ and are defined by
\begin{equation}
  {(D_1 u)}_{i,j} :=
  \begin{cases}%
    u_{i+1,j} - u_{i,j}  &   1 \le i < m, \\
    0                    &   \text{else,}
  \end{cases}
\end{equation}
and
\begin{equation}
  {(D_2 u)}_{i,j} :=
  \begin{cases}%
    u_{i,j+1} - u_{i,j}  &   1 \le j < m, \\
    0                    &   \text{else.}
  \end{cases}
\end{equation}
The operator norm of $D_1$ and $D_2$, respectively, is $2$ (where we equipped $\R^{m \times n}$ with the Frobenius norm).
This yields an operator norm of $\sqrt{8}$ for the total gradient $D := D_1 \times D_2$ as a map from $\R^{m \times n}$ to $\R^{m \times n} \times \R^{m \times n}$, see also~\cite{chambolle2004algorithm}.

We will compare our methods, i.e.\ the Variable Accelerated SmooThing (VAST) and its stochastic counterpart (sVAST) to the Primal Dual Hybrid Gradient (PDHG) of~\cite{pdhg} as well as its stochastic version (sPDHG) from~\cite{spdhg}. Furthermore, we will illustrate another competitor, the method by Pesquet and Repetti, see~\cite{pesquet_repetti}, which is another a stochastic version of PDHG (see also~\cite{vu}).

In all examples we choose the parameters in accordance with~\cite{spdhg}:
\begin{itemize}
  \item for PDHG and Pesquet\&Repetti: $\tau = \sigma_i = \frac{\gamma}{\lVert K \rVert}$
    \item for sPDHG:\ $\sigma_i = \frac{\gamma}{\lVert K \rVert}$ and $\tau = \frac{\gamma}{n \max_i \lVert K_i \rVert}$,
\end{itemize}
where $\gamma = 0.99$.

\subsection{Total Variation Denoising}%
\label{sub:total_variation_denoising}
The task at hand is to reconstruct an image from its noisy observation.
We do this by solving 
\begin{equation}
  \min_{x \in \R^{m \times n}{}} \alpha \lVert x - b  \rVert_2 + \lVert D_1 x \rVert_1 + \lVert D_2 x \rVert_1,
\end{equation}
with $\alpha >0$ as regularization parameter, in the following setting: $f= \alpha \lVert \cdot - b \rVert_2, g_1=g_2 = \lVert \cdot \rVert_1, K_1=D_1, K_2=D_2$. 

\begin{figure}
  \centering
  \begin{subfigure}[b]{0.3\linewidth}
    \centering
    \includegraphics[width=\linewidth]{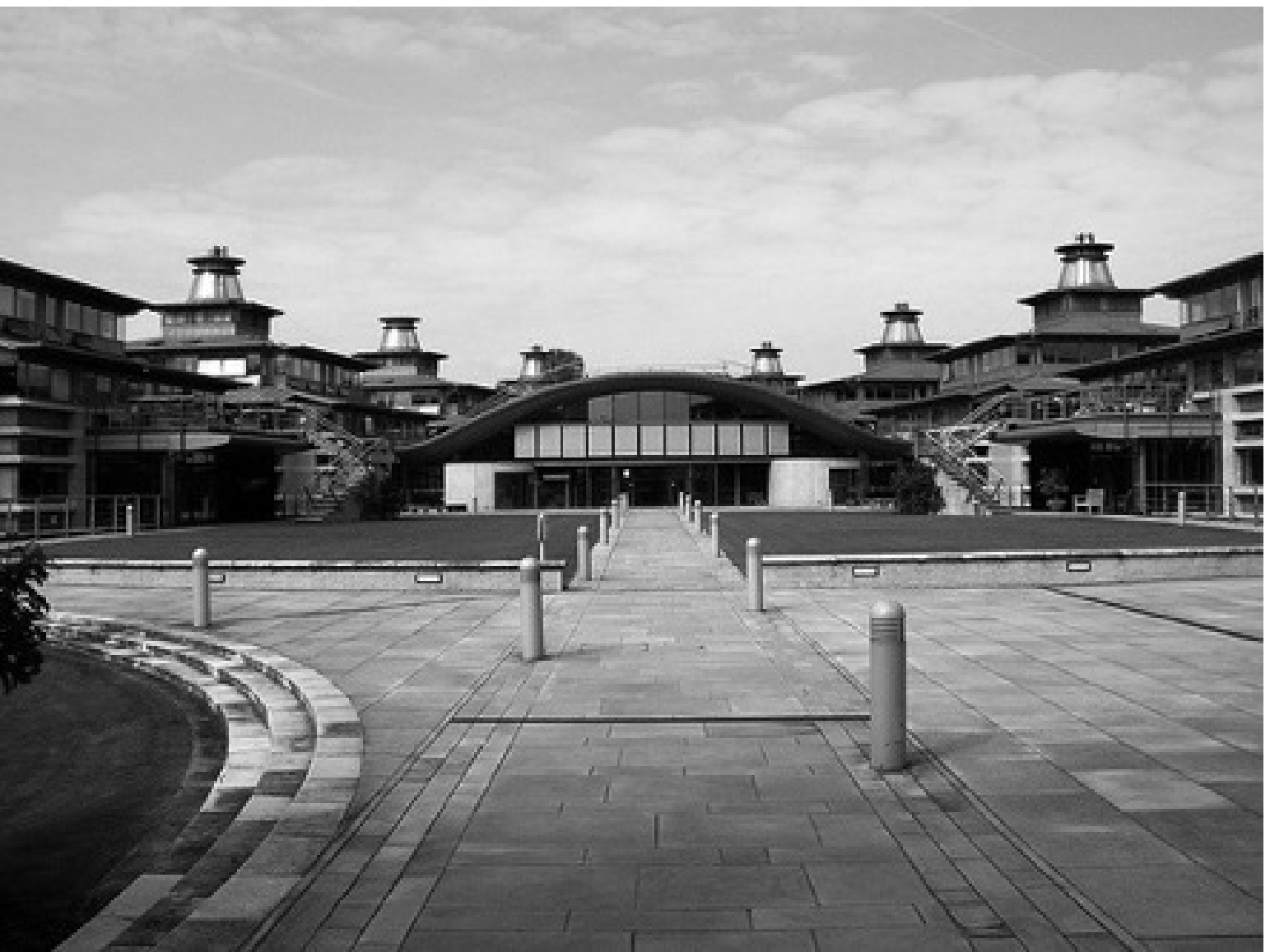}
    \caption{Groundtruth}
  \end{subfigure}
  \begin{subfigure}[b]{0.3\linewidth}
    \centering
    \includegraphics[width=\linewidth]{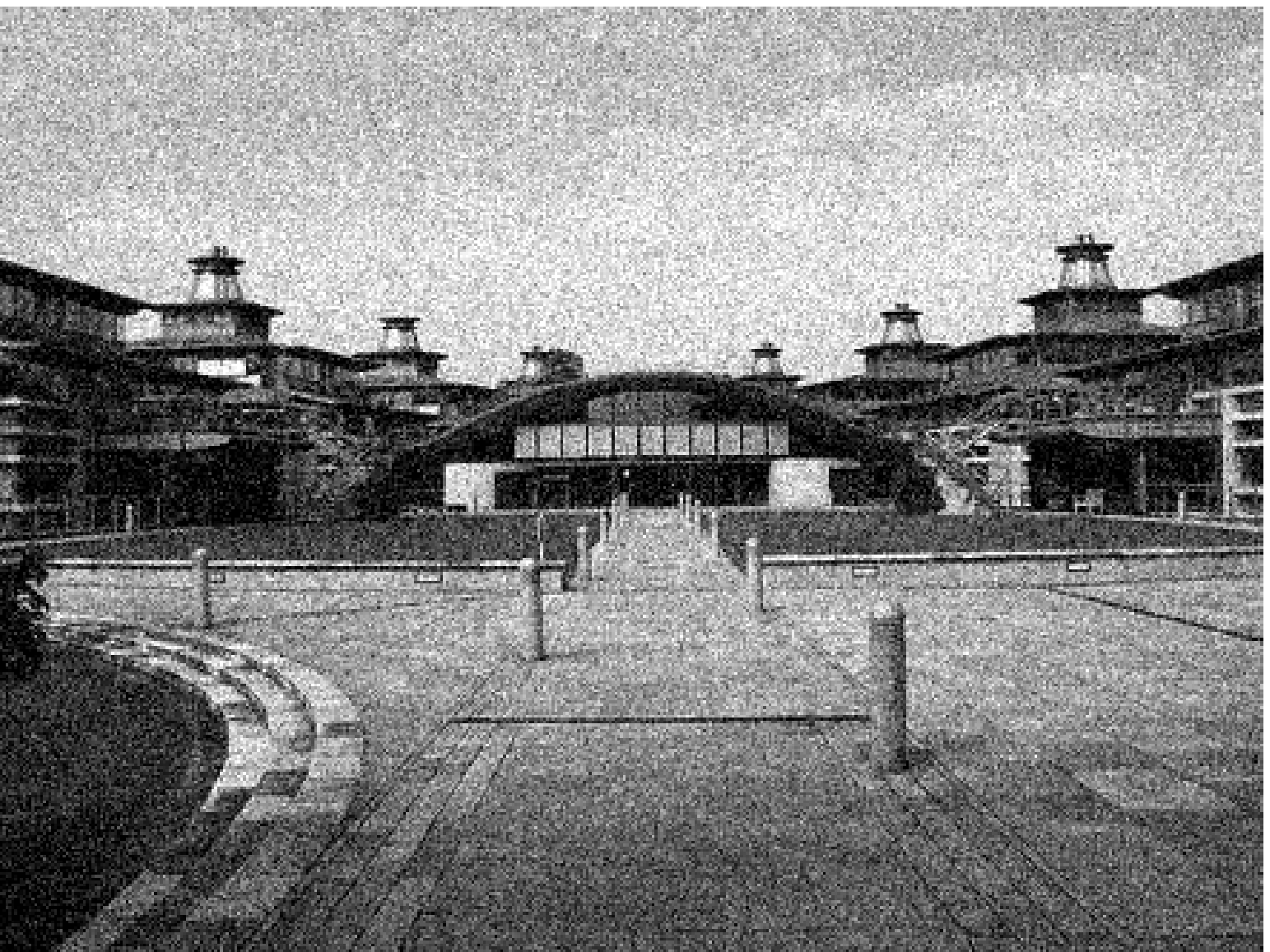}
    \caption{Data}
  \end{subfigure}
  \begin{subfigure}[b]{0.3\linewidth}
    \centering
    \includegraphics[width=\linewidth]{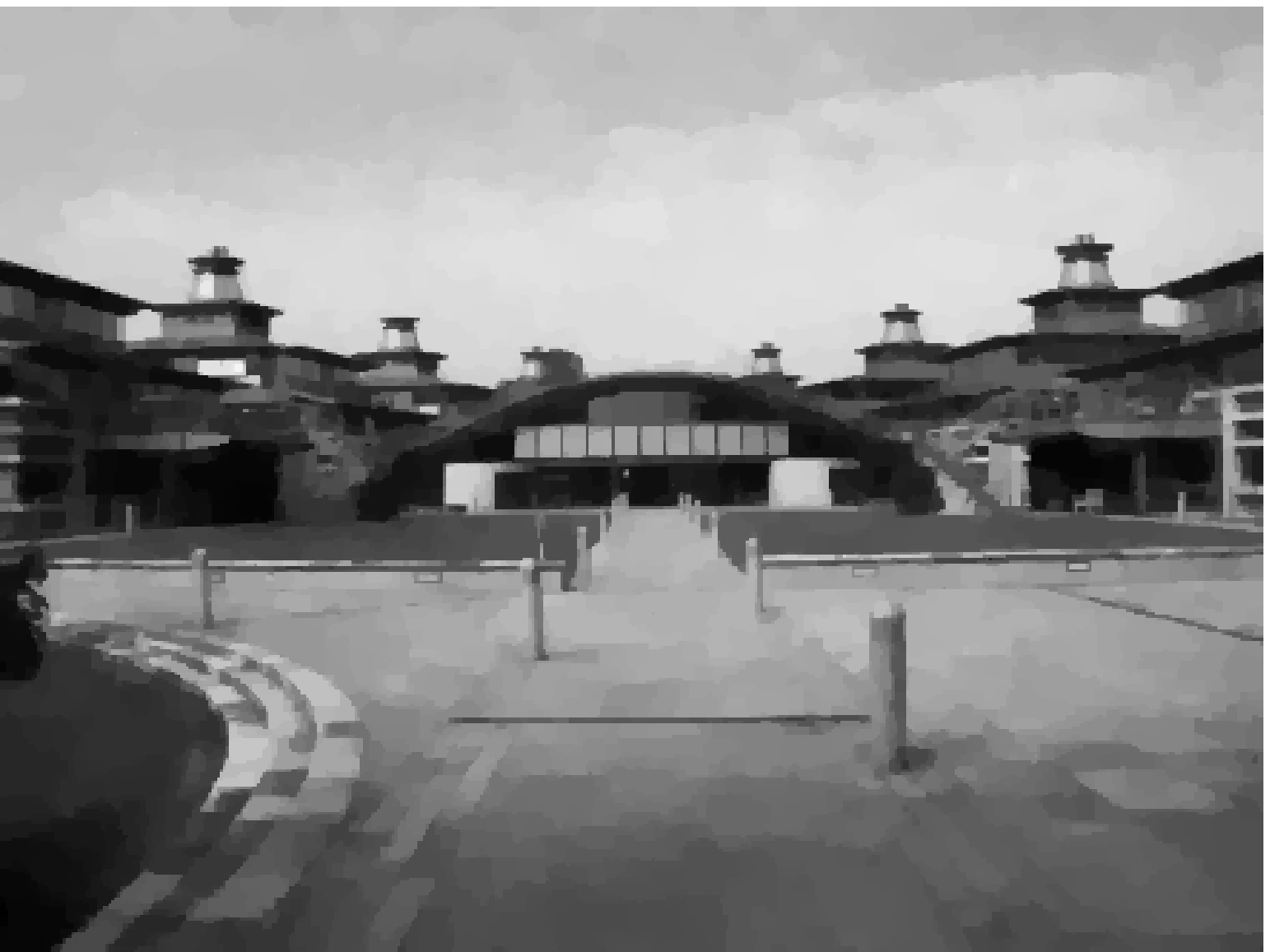}
    \caption{Approximate solution}
  \end{subfigure}
  \caption{TV denoising. Images used. The approximate solution is computed by running PDHG for 7000 iterations.}%
  \label{fig:images-used-denoising}
\end{figure}

Figure~\ref{fig:images-used-denoising} illustrates the images (of dimension $m = 442$ and $n = 331$) used in for this example. These include the groundtruth, i.e.\ the uncorrupted image, as well as the data for the optimization problem $b$, which visualizes the level of noise.
In Figure~\ref{fig:denoising-plots} we can see that for the deterministic setting our method is as good as PDHG.\ For the objective function values, Subfigure~\ref{fig:denoising-objfun}, this is not too surprising as both algorithms share the same convergence rate. For the distance to a solution however we completely lack a convergence result. Nevertheless in Subfigure~\ref{fig:denoising-dist} we can see that our method performs also well with respect to this measure.

\begin{figure}[]
  \centering
  \begin{subfigure}[b]{0.49\linewidth}%
    \includegraphics[width=\linewidth]{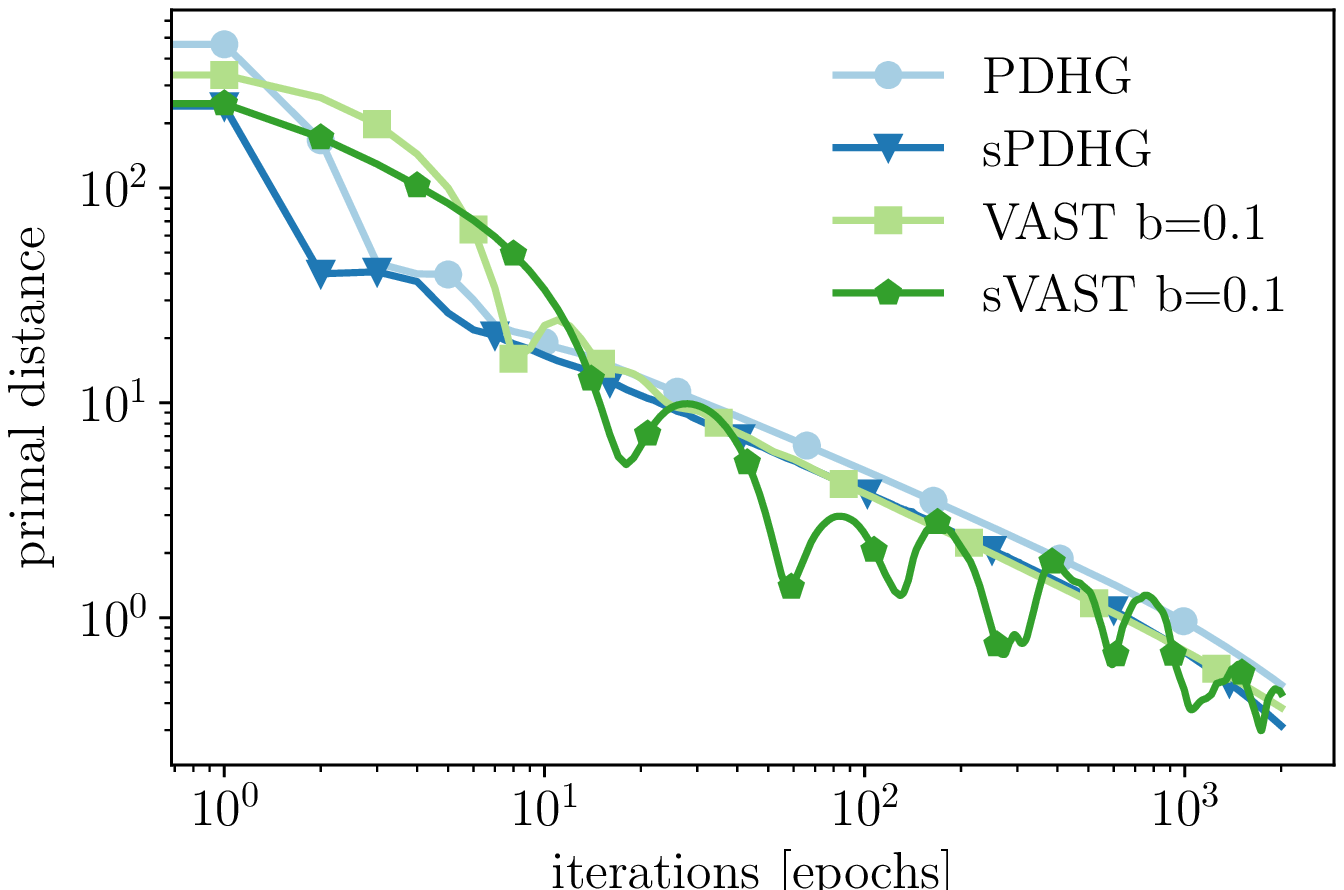}
    \caption{Distance to the solution.}%
    \label{fig:denoising-dist}
  \end{subfigure}
  \begin{subfigure}[b]{0.49\linewidth}%
    \includegraphics[width=\linewidth]{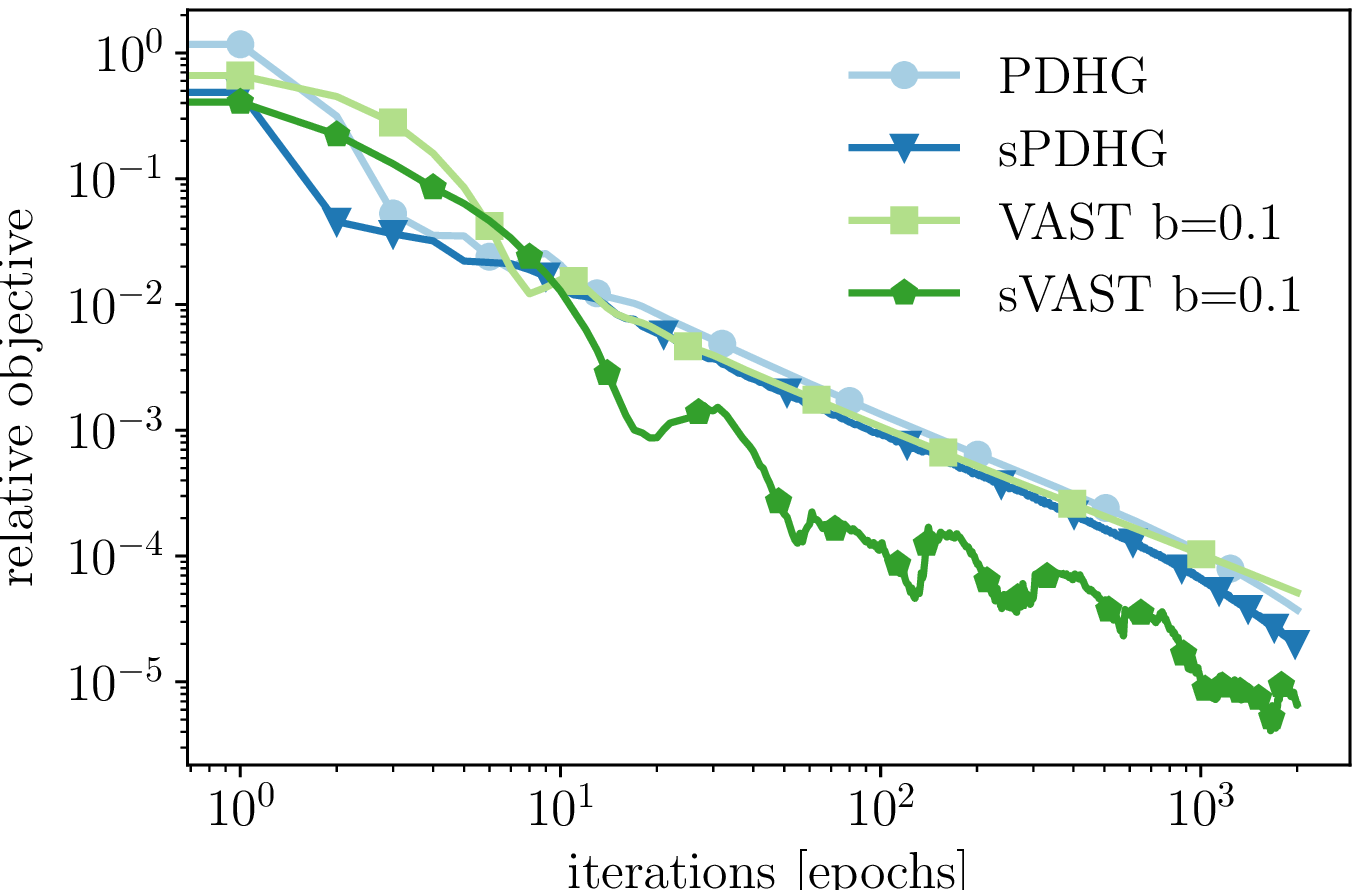}
    \caption{Relative objective $\frac{F(x_{k}) - F(x^*)}{F(x_0) - F(x^*)}$.}%
    \label{fig:denoising-objfun}
  \end{subfigure}
    \caption{TV denoising. Plots.}%
  \label{fig:denoising-plots}
\end{figure}

In the stochastic setting we can see in Figure~\ref{fig:denoising-plots} that, while sPDHG provides some benefit over its deterministic counterpart, the stochastic version of our method, although significantly increasing the variance, provides great benefit, at least for the objective function values.

Furthermore, Figure~\ref{fig:denoising-reconstruction}, shows the reconstructions of sPDHG and our method which are, despite the different objective function values, quite comparable.

\begin{figure}
  \centering
  \begin{subfigure}[b]{0.3\linewidth}
    \centering
    \includegraphics[width=\linewidth]{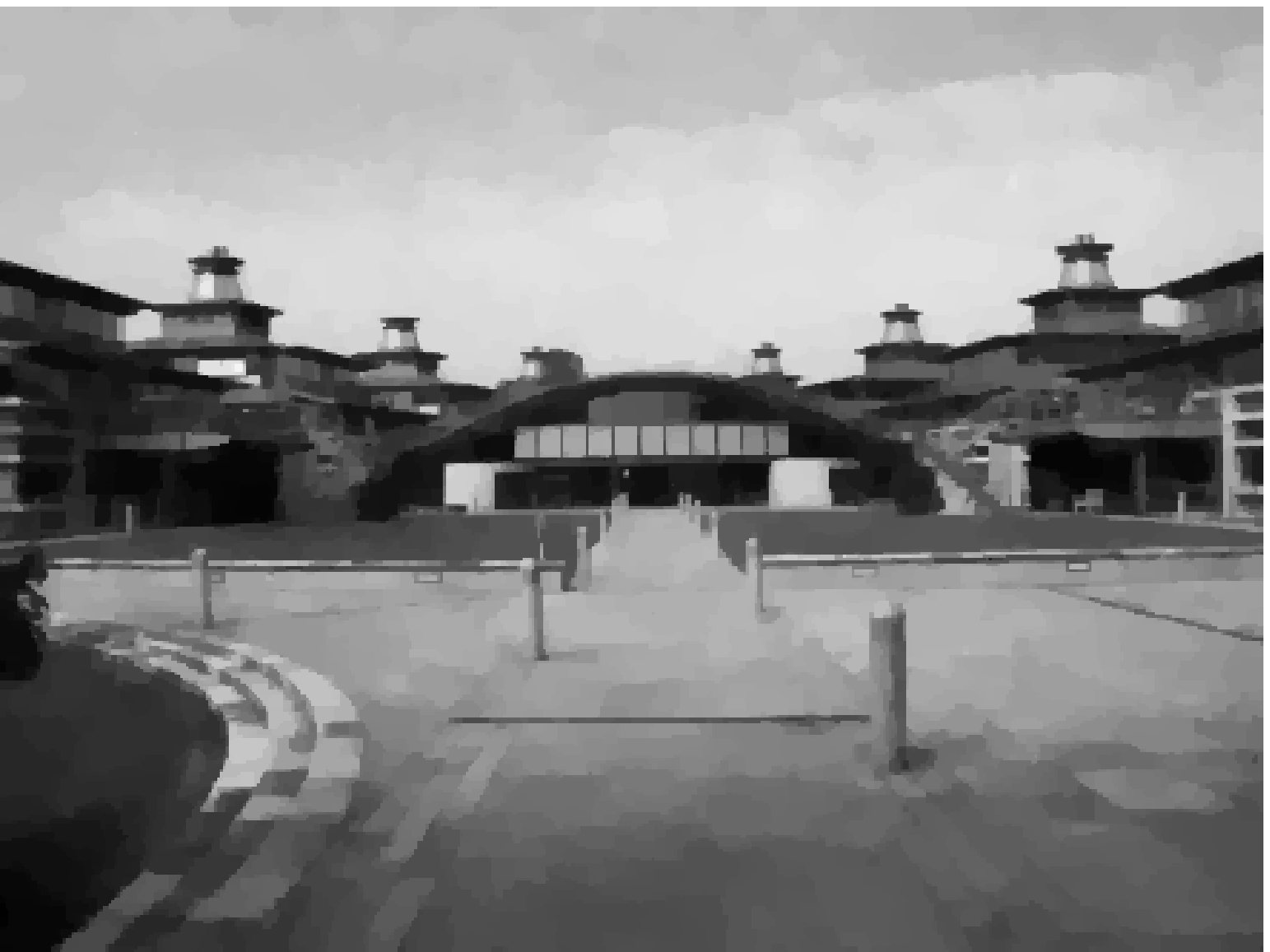}
    \caption{sVAST}
  \end{subfigure}
  \begin{subfigure}[b]{0.3\linewidth}
    \centering
    \includegraphics[width=\linewidth]{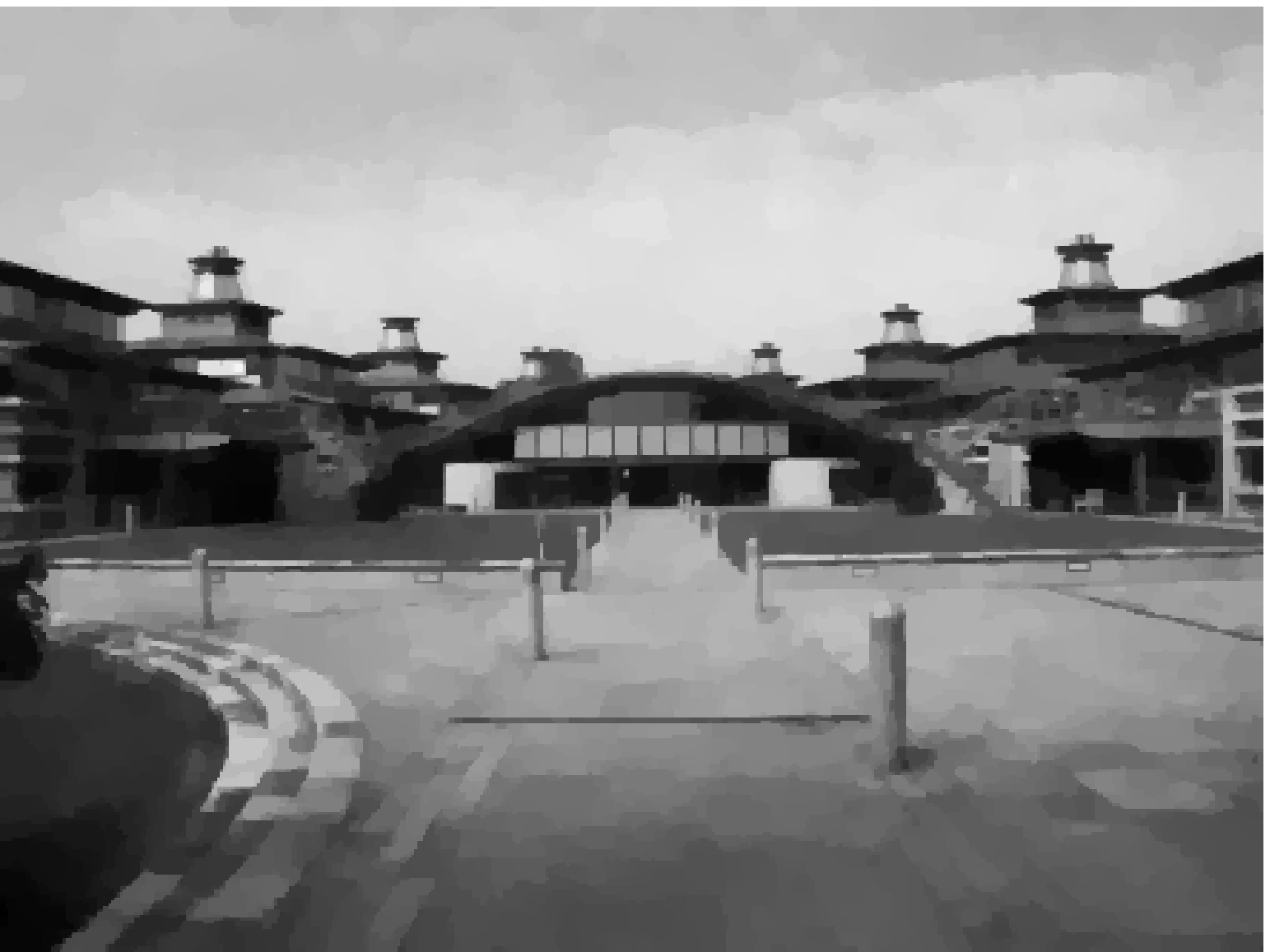}
    \caption{sPDHG}
  \end{subfigure}
  \caption{TV Denoising. A comparison of the reconstruction for the stochastic variable smoothing method and the stochastic PDHG.}%
  \label{fig:denoising-reconstruction}
\end{figure}

\subsection{Total Variation Deblurring}%
\label{sub:total_variation_deblurring}
For this example we want to reconstruct an image from a blurred and noisy image. We assume to know the blurring operator $C:  \R^{m \times n} \rightarrow  \R^{m \times n}$. This is done by solving
\begin{equation}
  \label{eq:deblurring-problem}
  \min_{x \in \R^{m \times n}{}}  \alpha \lVert C x - b  \rVert_2 + \lVert D_1 x \rVert_1 + \lVert D_2 x, \rVert_1,
\end{equation}
for $\alpha >0$ as regularization parameter, in the following setting: $f=0, g_1 = \alpha \lVert \cdot - b \rVert_2, g_2=g_3 = \lVert \cdot \rVert_1, K_1=C, K_2 =  D_1, K_2=D_2$.

Figure~\ref{fig:images-used-deblurring} shows the images used to set up the optimization problem~\eqref{eq:deblurring-problem}, in particular Subfigure~\ref{fig:deblurring-data} which corresponds to $b$ in said problem.

In Figure~\ref{fig:deblurring-plots} we see that while PDGH performs better in the deterministic setting, in particular in the later iteration, the stochastic variable smoothing method provides a significant improvement where sPDHG method seems not to converge. It is interesting to note that in this setting even the deterministic version of our algorithm exhibits a slightly chaotic behaviour. Although neither of the two methods is monotone in the primal objective function PDHG seems here much more stable.

\begin{figure}
  \centering
  \begin{subfigure}[b]{0.3\linewidth}
    \centering
    \includegraphics[width=\linewidth]{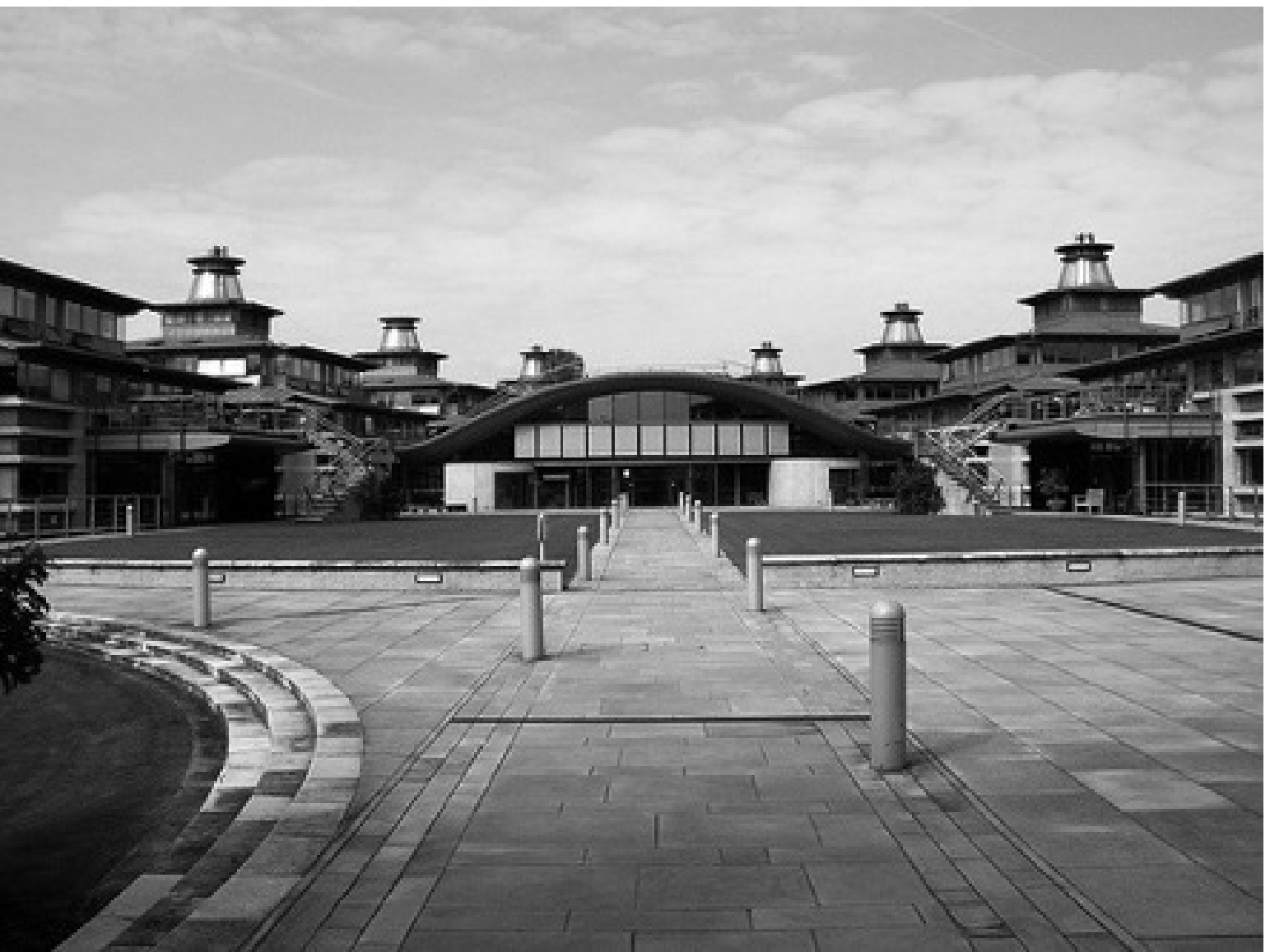}
    \caption{Groundtruth}
  \end{subfigure}
  \begin{subfigure}[b]{0.3\linewidth}
    \centering
    \includegraphics[width=\linewidth]{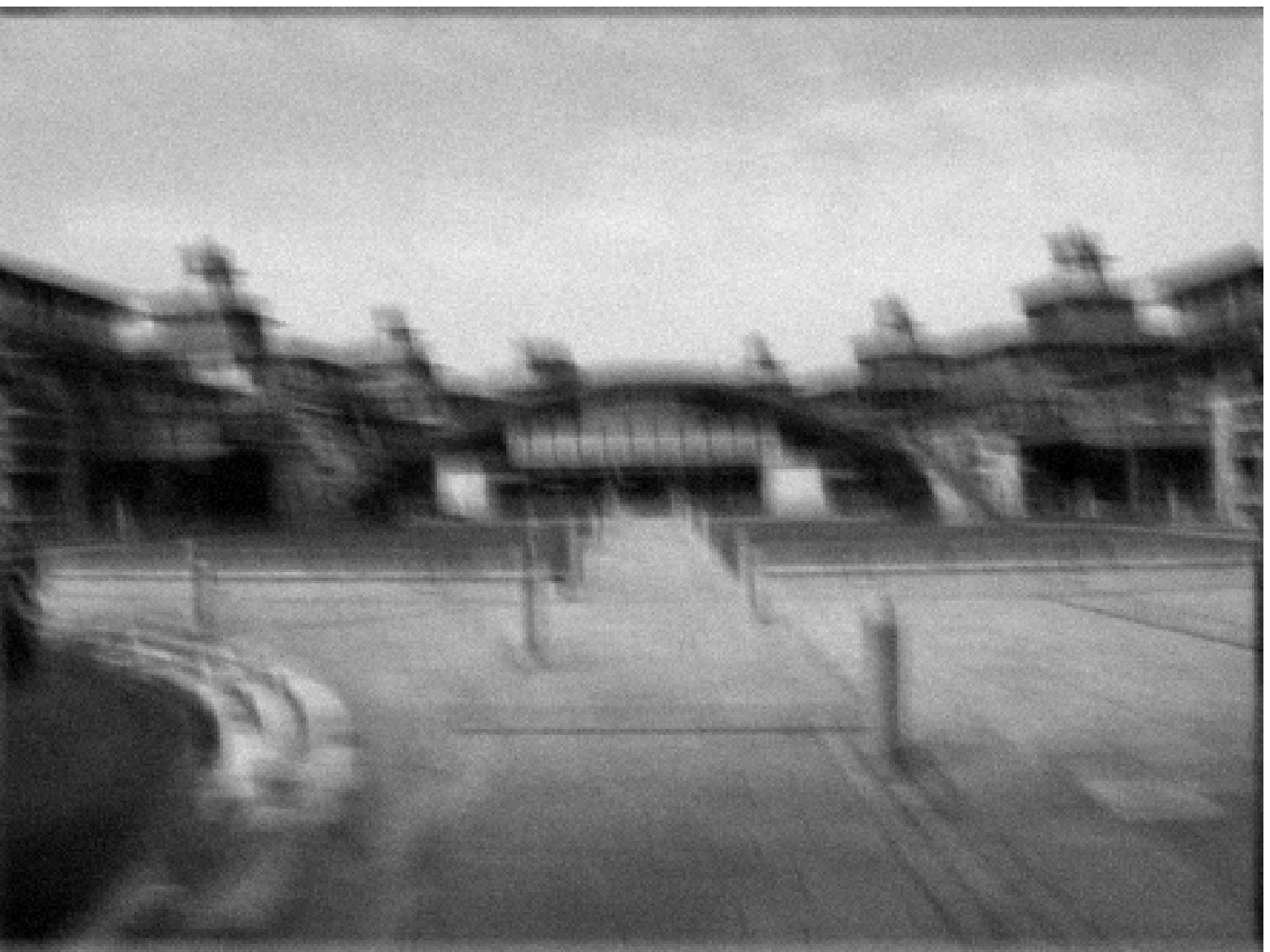}
    \caption{Data}%
    \label{fig:deblurring-data}
  \end{subfigure}
  \begin{subfigure}[b]{0.3\linewidth}
    \centering
    \includegraphics[width=\linewidth]{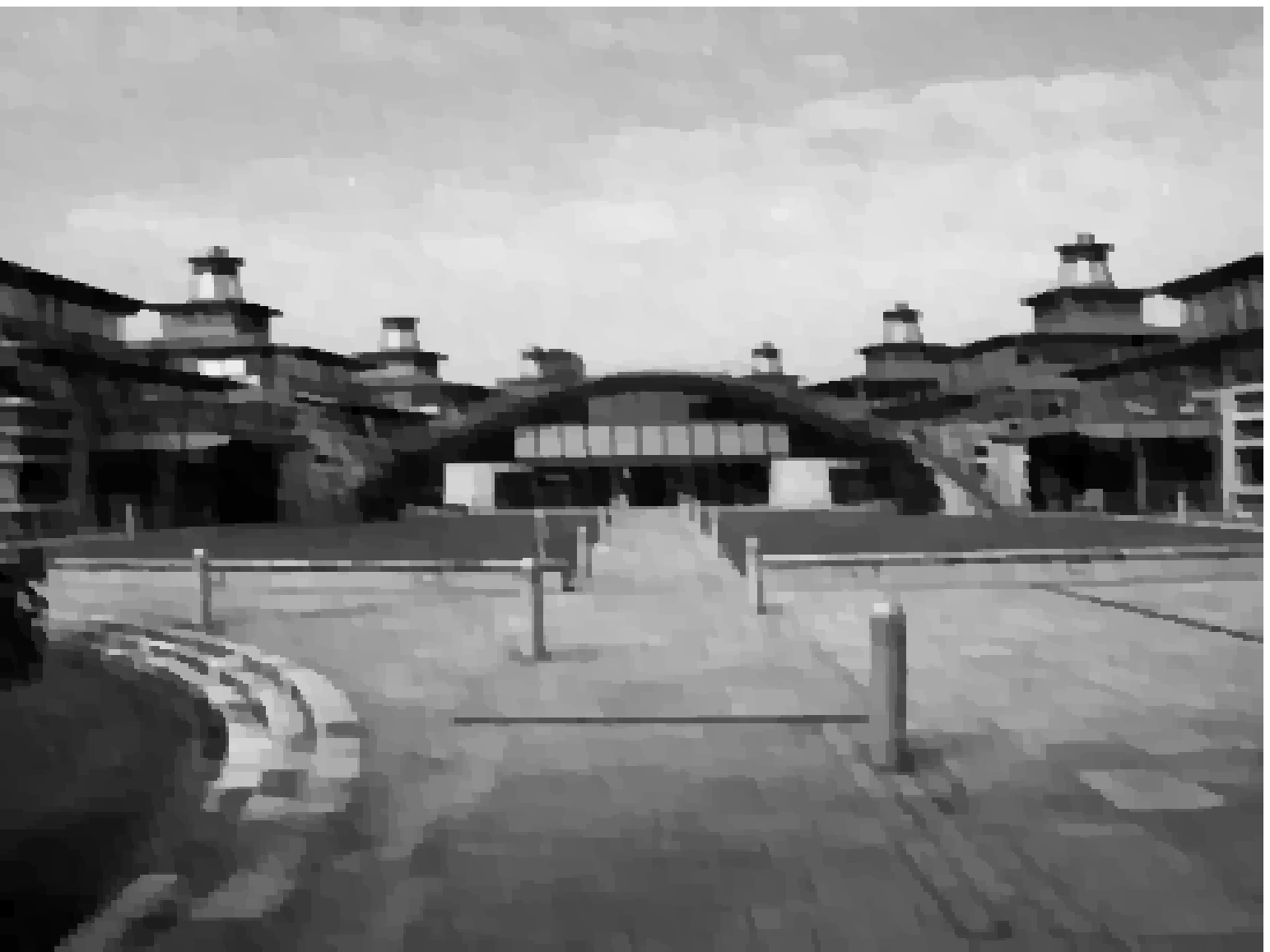}
    \caption{Approximate solution}
  \end{subfigure}
  \caption{TV Deblurring.The approximate solution is computed by running PDHG for 3000 iterations.}%
  \label{fig:images-used-deblurring}
\end{figure}

\begin{figure}[]%
  \label{fig:deblurring-funval}
  \centering
  \begin{subfigure}[b]{0.49\linewidth}
    \centering
    \includegraphics[width=\linewidth]{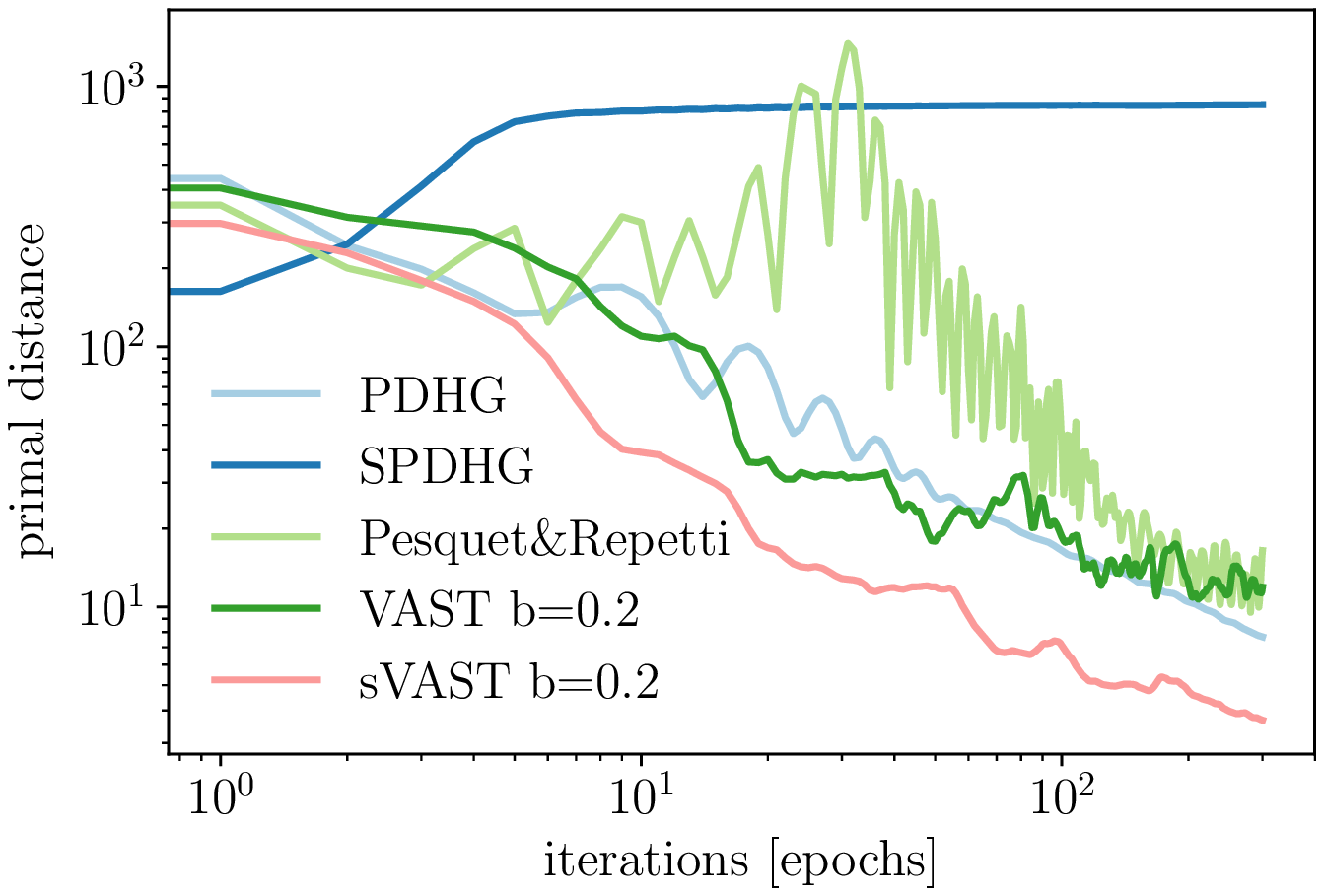}
    \caption{Distance to the solution.}%
    \label{fig:deblurring-dist}
  \end{subfigure}
  \begin{subfigure}[b]{0.49\linewidth}
    \centering
    \includegraphics[width=\linewidth]{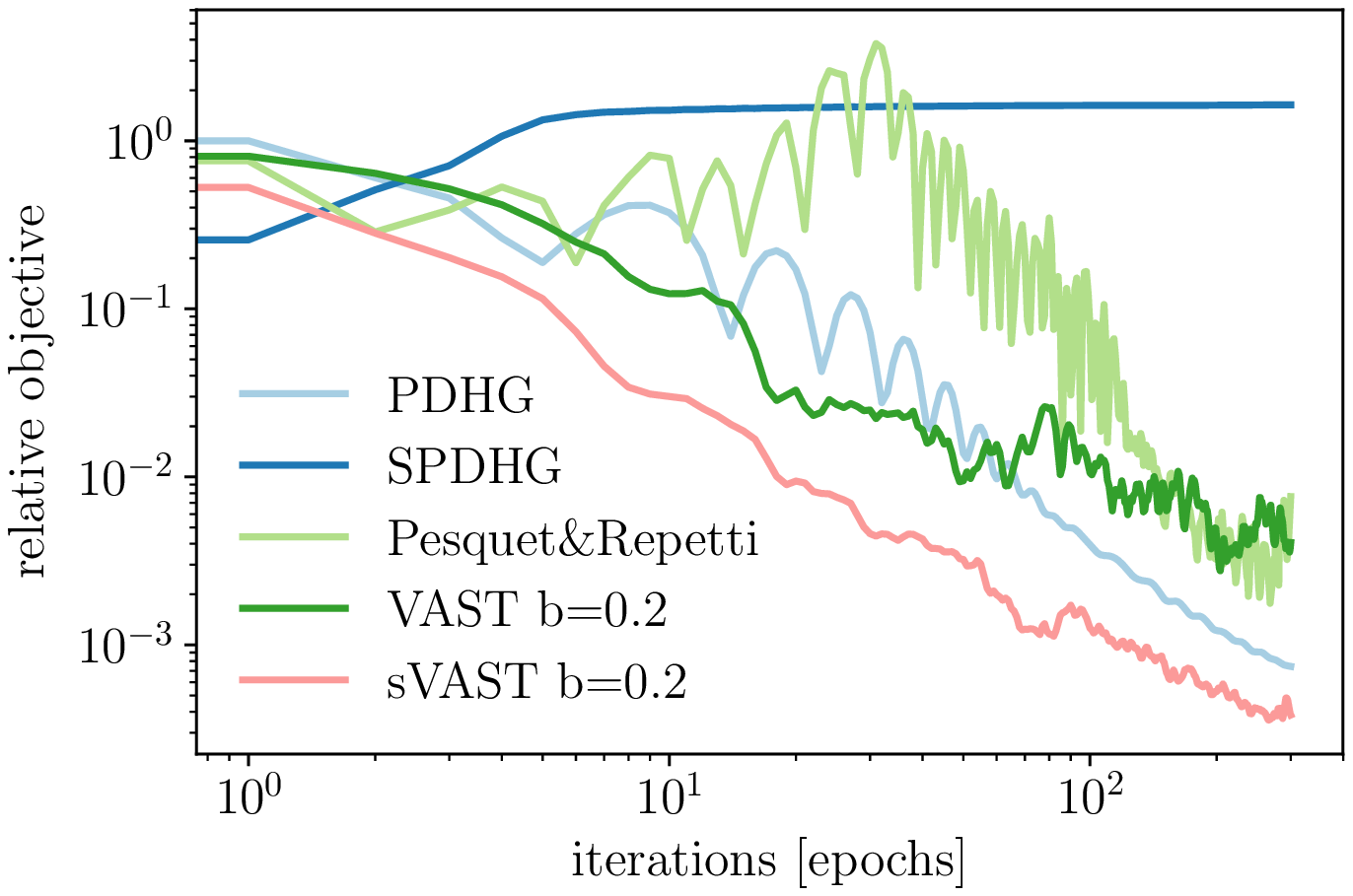}
    \caption{Relative objective $\frac{F(x_{k}) - F(x^*)}{F(x_0) - F(x^*)}$.}%
    \label{fig:deblurring-obfun}
  \end{subfigure}
  \caption{TV deblurring. Plots.}%
  \label{fig:deblurring-plots}
\end{figure}

\end{document}